\documentclass{amsart}
\usepackage{amsmath,amsthm,fixltx2e}
\usepackage{color}
\usepackage{amsxtra}
\usepackage{amscd}
\usepackage{amsfonts}
\usepackage{amssymb}
\usepackage{eucal}

\newtheorem{theorem}{Theorem}[section]
\newtheorem{cor}[theorem]{Corollary}
\newtheorem{lem}[theorem]{Lemma}
\newtheorem{prop}[theorem]{Proposition}

\newtheorem{thm}[theorem]{Theorem}

\newtheorem{rem}[theorem]{Remark}

\newtheorem{defn}[theorem]{Definition}

\newcommand{\nc}{\newcommand}

\nc\ol{\overline} \nc\ul{\underline} \nc\wt{\widetilde}
\nc{\z}{\zeta}

\nc{\ZZ}{{\mathbb Z}} \nc{\NN}{{\mathbb N}} \nc{\CC}{{\mathbb C}}
\nc{\QQ}{{\mathbb Q}} \nc{\CP}{{\mathbb {CP}}}

\nc{\A}{{\mathcal A}}  \nc\U{{\mathfrak U}}

\nc{\F}{{\mathcal F}} \nc{\N}{{\mathcal N}} \nc{\Aa}{{\mathcal A}}
\nc{\E}{{\mathcal E}} \nc{\sS}{{\mathbb S}} \nc{\K}{{\mathcal K}}
\nc{\Ll}{{\mathcal L}}

\newcommand\h{\mathfrak h}

\newcommand\aaa{\mathfrak a}

\newcommand{\gl}{\mathfrak{gl}}
\newcommand{\ssl}{\mathfrak{sl}}

\newcommand{\Sym}{\mathrm{Sym}}

\newcommand{\tr}{\mathrm{tr}}

\nc{\iso}{{\stackrel{\sim}{\longrightarrow}}}

\begin{document}

\author{Boris Feigin}
 \address{B.~Feigin:
  National Research University Higher School of Economics, Russian Federation,
  International Laboratory of Representation Theory and Mathematical Physics, Moscow, Russia}
 \email{borfeigin@gmail.com}

\author[Alexander Tsymbaliuk]{Alexander Tsymbaliuk}
 \address{A.~Tsymbaliuk: Simons Center for Geometry and Physics, Stony Brook, NY 11794, USA}
 \email{otsymbaliuk@scgp.stonybrook.edu}

\title[Bethe subalgebras of $U_q(\widehat{\gl}_n)$ via shuffle algebras]
  {Bethe subalgebras of $U_q(\widehat{\gl}_n)$ via shuffle algebras}

\begin{abstract}
  In this article, we construct certain commutative subalgebras of the \emph{big} shuffle algebra of type $A^{(1)}_{n-1}$.
 This can be considered as a generalization of the similar construction for the \emph{small} shuffle algebra, obtained in~\cite{FS}.
 We present a \emph{Bethe algebra} realization of these subalgebras. The latter identifies them with the Bethe subalgebras of $U_q(\widehat{\gl}_n)$.
\end{abstract}

\maketitle

\section*{Introduction}

  \emph{Elliptic} shuffle algebras were first introduced and studied by the first author and A.~Odesskii, see~\cite{FO1,FO2,FO3}.
 In the \emph{loc.cit.}, they were associated with an elliptic curve $\E$ endowed with two automorphisms $\tau_1,\tau_2$.
  A similar class of algebras, depending on two parameters (alternatively $q_1,q_2,q_3$ with $q_1q_2q_3=1$),
 became of interest in the recent years, due to their geometric
 interpretations and different algebraic incarnations (see~\cite{FS,FT,N1,SV} for the related results).
 We will refer to these algebras as the \emph{small} shuffle algebras.
   In this paper, we study the higher-rank generalizations of those algebras,
 which we refer to as the \emph{big} shuffle algebras (of $A^{(1)}_{n-1}$-type).
 These algebras were also recently considered in~\cite{N2}, where they
 were identified with the positive half of the quantum toroidal algebras $\ddot{U}_{q,d}(\ssl_n)$.

  The aim of this paper is to study particular \emph{large} commutative subalgebras of the \emph{big} shuffle
 algebra $S$, similar to the one from~\cite{FS}.
  We also establish a \emph{Bethe algebra} realization of these subalgebras (which seems to be new even for the \emph{small} shuffle algebras).
 In other words, we identify those commutative subalgebras with the standard Bethe subalgebras of the quantum affine algebra $U_q(\widehat{\gl}_n)$,
 which is \emph{horizontally} embedded into the quantum toroidal algebra.

                     %%%%%%%%%%%%%%%%%%%%%%%%%%%%%%%%%%%%%%%%%%%%%
                     %%%%%%%%%%%%%%%%%% REMARK # 1 %%%%%%%%%%%%%%%
                     %%%%%%%%%%%%%%%%%%%%%%%%%%%%%%%%%%%%%%%%%%%%%
  %  Actually, we should replace $U_q(\widehat{\gl}_n)$ by $U_q(L\gl_n)$ everywhere,
  % since eventually in Theorem 3.10, we identify our commutative shuffle subalgebras
  % with Bethe subalgebras of $U_q(L\gl_n)$, which is horizontally embedded into
  % the modification $\ddot{U}'_{q,d}(\ssl_n)$ of $\ddot{U}_{q,d}(\ssl_n)$ from Theorem 1.7(c).
  %  However, we feel awkward of changing the title of the paper and also mentioning this
  % modification of toroidal algebra in the Introduction.
  % So we leave as Bethe subalegbras of $U_q(\widehat{\gl}_n)$, though what we really mean
  % is Bethe subalegbras of $U_q(L\gl_n)$.
                     %%%%%%%%%%%%%%%%%%%%%%%%%%%%%%%%%%%%%%%%%%%%%
                     %%%%%%%%%%%%%%%%%%% END %%%%%%%%%%%%%%%%%%%%%
                     %%%%%%%%%%%%%%%%%%%%%%%%%%%%%%%%%%%%%%%%%%%%%

  The aforementioned commutative subalgebras of $S$ admit a one-parameter deformation: the commutative subalgebras
 $\A(s_0,\ldots,s_{n-1};t)\subset  (S^\geq)^\wedge$
 (the algebra $S^\geq$ is a slight enhancement of $S$, see Section 3.3, while $^\wedge$ indicates the completion with respect to the natural $\ZZ$-grading).
  These algebras are closely related to the study of nonlocal integrals of motion for the deformed $W$-algebras $W_{q,t}(\widehat{\ssl}_n)$
 from~\cite{FKSW}, as well as provide a framework for the generalization of the recent
 results from~\cite{FJMM2} to $\ddot{U}_{q,d}(\ssl_n)$.
 This will be elaborated elsewhere.

 This paper is organized as follows:

 $\bullet$
  In Section 1, we recall the definition and key results about the
 quantum toroidal algebra $\ddot{U}_{q,d}(\ssl_n),\ n\geq 3$. We also recall the notion of
 the \emph{small} shuffle algebra $S^{\mathrm{sm}}$ and its commutative subalgebra $\A^{\mathrm{sm}}$,
 and introduce a higher-rank generalization, the \emph{big} shuffle algebra $S$.

 $\bullet$
  In Section 2, we introduce a family of subspaces $\A(s_0,\ldots,s_{n-1})\subset S$ depending on $n$ parameters
 and  generalizing the construction of $\A^{\mathrm{sm}}\subset S^{\mathrm{sm}}$.
  If $(\frac{1}{s_1\ldots s_{n-1}},s_1,\ldots,s_{n-1})$ is \emph{generic} (see Section 2.2),
 then we prove that $\A(\frac{1}{s_1\ldots s_{n-1}},s_1,\ldots,s_{n-1})$ is a polynomial algebra on
 explicitly given generators; in particular, it is a commutative subalgebra of $S$.

 $\bullet$
  In Section 3, we use the universal $R$-matrix and vertex-type
 representations to establish an alternative viewpoint toward $\A(s_0,\ldots,s_{n-1})$.
 This allows us to identify them with the well-known Bethe subalgebras of the quantum affine $U_q(\widehat{\gl}_n)$,
 horizontally embedded into $\ddot{U}_{q,d}(\ssl_n)$.

 $\bullet$
  In Section 4, we discuss generalizations of the results from Sections 1-3 to the cases $n=1,2$.

%%%%%%%%%%%%%%%%%%%%%%%%%%%%%%%%%%%%%%%%%%%%%%%% Acknowledgments %%%%%%%%%%%%%%%%%%%%%%%%%%%%%%%%%%%%%%%%%%%%%%%%%%%%%%%%%%%

\subsection*{Acknowledgments}
$\ $

  We are grateful to A.~Negut and J.~Shiraishi for stimulating discussions.
 We are indebted to B.~Enriquez for useful comments on the first
 version of the paper, which led to a better exposition of the material.
  A.~T. is grateful to P.~Etingof and H.~Nakajima for their interest and support.
  A.~T. thanks the Research Institute for Mathematical Sciences
 (Kyoto) and the Japan Society for the Promotion of Science for
 support during the main stage of the project.
  A.~T. also gratefully acknowledges support from the Simons Center for Geometry and Physics, Stony Brook University,
 at which some of the research for this paper was performed.

  The work of A.~T. was partially supported by the NSF Grant DMS-1502497.
  B.~F. gratefully acknowledges the financial support of a subsidy granted to the HSE by the Government
 of the Russian Federation for the implementation of the Global Competitiveness Program.

%%%%%%%%%%%%%%%%%%%%%%%%%%%%%%%%%%%%%%%%%%%%%%%%%%% SECTION 1 %%%%%%%%%%%%%%%%%%%%%%%%%%%%%%%%%%%%%%%%%%%%%%%%%%%%%%%%%%%%%%%%%%%%%%

\section{Basic definitions and constructions}

                                 %%%%%%%%%%%%%% Quantum Toroidal %%%%%%%%%%%%%%

\subsection{Quantum toroidal algebras of $\ssl_n$ for $n\geq 3$}
 $\ $

  Let $q,d\in \CC^*$ be two parameters. We set $[n]:=\{0,1,\ldots,n-1\},\ [n]^\times:=[n]\backslash \{0\}$,
 the former viewed as a set of mod $n$ residues.
 Let $g_m(z):=\frac{q^mz-1}{z-q^m}$. Define $\{a_{i,j},m_{i,j}|i,j\in [n]\}$ by
  $$a_{i,i}=2,\ a_{i,i\pm 1}=-1,\ m_{i,i\pm1}=\mp 1,\ \mathrm{and}\ a_{i,j}=m_{i,j}=0\ \mathrm{otherwise}.$$
  The quantum toroidal algebra of $\ssl_n$, denoted by $\ddot{U}_{q,d}(\ssl_n)$,
 is the unital associative algebra generated by
   $\{e_{i,k}, f_{i,k}, \psi_{i,k}, \psi_{i,0}^{-1}, \gamma^{\pm 1/2}, q^{\pm d_1}, q^{\pm d_2}\}_{i\in [n]}^{k\in \ZZ}$
 with the following defining relations:\footnote{\   Our notation are consistent with that of~\cite{VV}, but following~\cite{S} we add
 the elements $q^{\pm d_1}, q^{\pm d_2}$ satisfying~(\ref{T0.3}, \ref{T0.4}). This update is essential for
 our discussion of the Drinfeld double and the universal $R$-matrix.
}
\begin{equation}\tag{T0.1} \label{T0.1}
  [\psi_i^\pm(z),\psi_j^\pm(w)]=0,\ \gamma^{\pm 1/2}-\mathrm{central},
\end{equation}
\begin{equation}\tag{T0.2} \label{T0.2}
 \psi_{i,0}^{\pm 1}\cdot \psi_{i,0}^{\mp 1}=\gamma^{\pm 1/2}\cdot \gamma^{\mp 1/2}=q^{\pm d_1}\cdot q^{\mp d_1}= q^{\pm d_2}\cdot q^{\mp d_2}=1,
\end{equation}
\begin{equation}\tag{T0.3}\label{T0.3}
 q^{d_1}e_i(z)q^{-d_1}=e_i(qz),\ q^{d_1}f_i(z)q^{-d_1}=f_i(qz),\ q^{d_1}\psi^\pm_i(z)q^{-d_1}=\psi^\pm_i(qz),
\end{equation}
\begin{equation}\tag{T0.4}\label{T0.4}
 q^{d_2}e_i(z)q^{-d_2}=qe_i(z),\ q^{d_2}f_i(z)q^{-d_2}=q^{-1}f_i(z),\ q^{d_2}\psi^\pm_i(z)q^{-d_2}=\psi^\pm_i(z),
\end{equation}
\begin{equation}\tag{T1} \label{T1}
  g_{a_{i,j}}(\gamma^{-1}d^{m_{i,j}}z/w)\psi_i^+(z)\psi_j^-(w)=g_{a_{i,j}}(\gamma d^{m_{i,j}}z/w)\psi_j^-(w)\psi_i^+(z),
\end{equation}
\begin{equation}\tag{T2} \label{T2}
  e_i(z)e_j(w)=g_{a_{i,j}}(d^{m_{i,j}}z/w)e_j(w)e_i(z),
\end{equation}
\begin{equation}\tag{T3}\label{T3}
  f_i(z)f_j(w)=g_{a_{i,j}}(d^{m_{i,j}}z/w)^{-1}f_j(w)f_i(z),
\end{equation}
\begin{equation}\tag{T4}\label{T4}
  (q-q^{-1})[e_i(z),f_j(w)]=\delta_{i,j}\left(\delta(\gamma w/z)\psi_i^+(\gamma^{1/2}w)-\delta(\gamma z/w)\psi_i^-(\gamma^{1/2}z)\right),
\end{equation}
\begin{equation}\tag{T5}\label{T5}
  \psi_i^\pm(z)e_j(w)=g_{a_{i,j}}(\gamma^{\pm 1/2}d^{m_{i,j}}z/w)e_j(w)\psi_i^\pm(z),
\end{equation}
\begin{equation}\tag{T6}\label{T6}
  \psi_i^{\pm}(z)f_j(w)=g_{a_{i,j}}(\gamma^{\mp 1/2}d^{m_{i,j}}z/w)^{-1}f_j(w)\psi_i^\pm(z),
\end{equation}
\begin{equation}\tag{T7.1}\label{T7.1}
  \underset{z_1,z_2}\Sym\ [e_i(z_1),[e_i(z_2),e_{i\pm  1}(w)]_q]_{q^{-1}}=0,\
  [e_i(z),e_j(w)]=0\ \mathrm{for}\ j\ne i,i\pm 1,
\end{equation}
\begin{equation}\tag{T7.2}\label{T7.2}
  \underset{z_1,z_2}\Sym\ [f_i(z_1),[f_i(z_2),f_{i\pm  1}(w)]_q]_{q^{-1}}=0,\
  [f_i(z),f_j(w)]=0\ \mathrm{for}\ j\ne i,i\pm 1,
\end{equation}
 where we set $[a,b]_x:=ab-x\cdot ba$ and define the generating series as follows:
  $$e_i(z):=\sum_{k=-\infty}^{\infty}{e_{i,k}z^{-k}},
    f_i(z):=\sum_{k=-\infty}^{\infty}{f_{i,k}z^{-k}},
    \psi_i^{\pm}(z):=\psi_{i,0}^{\pm 1}+\sum_{r>0}{\psi_{i,\pm r}z^{\mp r}},
    \delta(z):=\sum_{k=-\infty}^{\infty}{z^k}.$$

  It will be convenient to use the generators $\{h_{i,k}\}_{k\ne 0}$ instead of  $\{\psi_{i,k}\}_{k\ne 0}$, defined by
  $$\exp\left(\pm(q-q^{-1})\sum_{r>0}h_{i,\pm r}z^{\mp r}\right)=\bar{\psi}_i^\pm(z):=\psi_{i,0}^{\mp 1}\psi^\pm_i(z),\ \
     h_{i,\pm r}\in \CC[\psi_{i,0}^{\mp 1},\psi_{i,\pm 1},\psi_{i,\pm2}, \ldots].$$
 Then the relations (T5,T6) are equivalent to the following (we use  $[m]_q:=(q^m-q^{-m})/(q-q^{-1})$):
\begin{equation}\tag{T5$'$}\label{T5'}
  \psi_{i,0}e_{j,l}=q^{a_{i,j}}e_{j,l}\psi_{i,0},\ [h_{i, k}, e_{j,l}]=d^{-km_{i,j}}\gamma^{-|k|/2}\frac{[ka_{i,j}]_q}{k}e_{j,l+k}\ (k\ne 0),
\end{equation}
\begin{equation}\tag{T6$'$}\label{T6'}
  \psi_{i,0}f_{j,l}=q^{-a_{i,j}}f_{j,l}\psi_{i,0},\ [h_{i, k},  f_{j,l}]=-d^{-km_{i,j}}\gamma^{|k|/2}\frac{[ka_{i,j}]_q}{k}f_{j,l+k}\ (k\ne 0).
\end{equation}
  We also introduce $h_{i,0},c,c'$ via $\psi_{i,0}=q^{h_{i,0}}, \gamma^{1/2}=q^c, c'=\sum_{i\in [n]} h_{i,0}$,
 so that $c,c'$ are central.

  Let $\ddot{U}^-$ and $\ddot{U}^+$ be the subalgebras of
 $\ddot{U}_{q,d}(\ssl_n)$ generated by $\{e_{i,k}\}_{i\in [n]}^{k\in \ZZ}$ and $\{f_{i,k}\}_{i\in [n]}^{k\in \ZZ}$, respectively,
 while $\ddot{U}^0$ is generated by $\{\psi_{i,k},\psi_{i,0}^{-1},\gamma^{\pm 1/2}, q^{\pm d_1}, q^{\pm d_2}\}_{i\in [n]}^{k\in \ZZ}$.

\begin{prop}\cite{H}(Triangular decomposition)\label{triangular_decomposition}
  The multiplication map
   $$m:\ddot{U}^-\otimes \ddot{U}^0\otimes \ddot{U}^+\to \ddot{U}_{q,d}(\ssl_n)$$
 is an isomorphism of vector spaces.
\end{prop}

  We equip the algebra $\ddot{U}_{q,d}(\ssl_n)$ with the $\ZZ^{[n]}\times \ZZ$-grading by assigning
 $$\deg(e_{i,k}):=(1_i;k),\ \deg(f_{i,k}):=(-1_i;k),\ \deg(\psi_{i,k}):=(0;k),$$
 $$\deg(x):=(0;0)\ \mathrm{for}\ x=\psi_{i,0}^{-1}, \gamma^{\pm 1/2}, q^{\pm d_1}, q^{\pm d_2}\ \ \forall\ i\in [n], k\in \ZZ,$$
  where $1_j\in \ZZ^{[n]}$ is the vector with the $j$th coordinate $1$ and all other coordinates being zero.

                            %%%%%%%%%%%%%% Vertical and horizontal affine of gl_n %%%%%%%%%%%%%%

\subsection{Horizontal and vertical $U_q(\widehat{\gl}_n)$}
$\ $

  Following~\cite{VV}, we introduce the \emph{vertical} and \emph{horizontal} copies of the
 quantum affine algebra of $\ssl_n$, denoted by $U_q(\widehat{\ssl}_n)$, inside $\ddot{U}_{q,d}(\ssl_n)$.
  Consider the subalgebra $\dot{U}^{\mathrm{v}}(\ssl_n)$ of $\ddot{U}_{q,d}(\ssl_n)$
 generated by $\{e_{i,k},f_{i,k},\psi_{i,k},\psi_{i,0}^{-1}, \gamma^{\pm 1/2}, q^{\pm d_1}|i\in [n]^\times,  k\in \ZZ\}$.
  This algebra is isomorphic to $U_q(\widehat{\ssl}_n)$, realized via the ``new Drinfeld presentation''.
  Let $\dot{U}^{\mathrm{h}}(\ssl_n)$ be the subalgebra of $\ddot{U}_{q,d}(\ssl_n)$
 generated by $\{e_{i,0},f_{i,0},\psi_{i,0}^{\pm 1}, q^{\pm d_2}|i\in [n]\}$.
 This algebra is also isomorphic to $U_q(\widehat{\ssl}_n)$, realized via the classical Drinfeld--Jimbo presentation.

  Following~\cite{FJMM1}, we recall a slight upgrade of this construction, which provides two copies of the quantum affine algebra of $\gl_n$,
 rather than $\ssl_n$, inside $\ddot{U}_{q,d}(\ssl_n)$.
  For every $r\ne 0$, choose $\{c_{i,r}|i\in [n]\}$ to be a nontrivial solution of the following system of linear equations:\footnote{\ It is easy to see that the
  space of solutions of this system is $1$-dimensional if $q$ is not a root of unity.}
   $$\sum_{i=0}^{n-1}c_{i,r}d^{-rm_{i,j}}[ra_{i,j}]_q =0,\ j\in [n]^\times.$$
 Let $\h^{\mathrm{v}}$ be the subspace of $\ddot{U}_{q,d}(\ssl_n)$ spanned by
\[
  h_r^{\mathrm{v}}=
   \begin{cases}
    \sum_{i\in [n]}c_{i,r}h_{i,r} & \text{if } r\ne 0 \\
    \gamma^{1/2} & \text{if } r=0
   \end{cases}.
\]
  Note that $\h^{\mathrm{v}}$ is well defined and commutes with $\dot{U}^{\mathrm{v}}(\ssl_n)$, due to (T5$'$, T6$'$).
 Moreover, $\h^{\mathrm{v}}$ is isomorphic to the Heisenberg Lie algebra.
  Let $\dot{U}^{\mathrm{v}}(\gl_n)$ be the subalgebra of $\ddot{U}_{q,d}(\ssl_n)$ generated by $\dot{U}^{\mathrm{v}}(\ssl_n)$ and $\h^{\mathrm{v}}$.
 The above discussions imply that $\dot{U}^{\mathrm{v}}(\gl_n)\simeq U_q(\widehat{\gl}_n)$, the quantum affine algebra of $\gl_n$.
 We let $\dot{U}^{\mathrm{v}}(\gl_1)\subset \dot{U}^{\mathrm{v}}(\gl_n)$ be the subalgebra generated by $\h^{\mathrm{v}}$.

  Our next goal is to provide a \emph{horizontal} copy of $U_q(\widehat{\gl}_n)$, containing $\dot{U}^{\mathrm{h}}(\ssl_n)$, inside $\ddot{U}_{q,d}(\ssl_n)$.
 The following approach was proposed in~\cite{FJMM1}, and it is based on a beautiful result of Miki:

\begin{thm}\cite{M}
  There exists an automorphism $\pi$ of $\ddot{U}_{q,d}(\ssl_n)$ such that
 $$\pi(\dot{U}^{\mathrm{v}}(\ssl_n))=\dot{U}^{\mathrm{h}}(\ssl_n),\ \pi(\dot{U}^{\mathrm{h}}(\ssl_n))=\dot{U}^{\mathrm{v}}(\ssl_n).$$
  Moreover:
 $$\pi(q^c)=q^{c'},\ \pi(q^{c'})=q^{-c}.$$
\end{thm}

                     %%%%%%%%%%%%%%%%%%%%%%%%%%%%%%%%%%%%%%%%%%%%%
                     %%%%%%%%%%%%%%%%%% REMARK #2 %%%%%%%%%%%%%%%%
                     %%%%%%%%%%%%%%%%%%%%%%%%%%%%%%%%%%%%%%%%%%%%%
  %  Actually this theorem is wrong as stated. Namely, the Miki's result applies to the toroidal algebra without q^{d_1} and q^{d_2}.
  % In the presence of the latter two elements, there is no such an automorphism, but rather a homomorphism (almost isomorphism)
  % between the two modifications (without d_1 or d_2).
  %  However, one can easily fix this to get a horizontal affine gl_n in our ambient toroidal...
  % Since we will never actually need this, and just mentioned this trick from [FJMM1] to address one of
  % the possible solutions, we do not bother to correct the statement appropriately.
                     %%%%%%%%%%%%%%%%%%%%%%%%%%%%%%%%%%%%%%%%%%%%%
                     %%%%%%%%%%%%%%%%%%% END %%%%%%%%%%%%%%%%%%%%%
                     %%%%%%%%%%%%%%%%%%%%%%%%%%%%%%%%%%%%%%%%%%%%%

  Let us define $\h^{\mathrm{h}}:=\pi(\h^{\mathrm{v}})$ and let $\dot{U}^{\mathrm{h}}(\gl_n)$ be the subalgebra of $\ddot{U}_{q,d}(\ssl_n)$
 generated by $\dot{U}^{\mathrm{h}}(\ssl_n)$ and $\h^{\mathrm{h}}$. Then $\dot{U}^{\mathrm{h}}(\gl_n)=\pi(\dot{U}^{\mathrm{v}}(\gl_n))$
 and it is isomorphic to $U_q(\widehat{\gl}_n)$. We also define
  $\dot{U}^{\mathrm{h}}(\gl_1)\subset  \dot{U}^{\mathrm{h}}(\gl_n)$ as the subalgebra generated by $\h^{\mathrm{h}}$.

  However, this construction is not very enlightening, as the images $\pi(h^{\mathrm{v}}_r)$ are hardly computable.
 An alternative approach, based on the RTT realization of $U_q(\widehat{\gl}_n)$, was proposed in~\cite{N2}.
 We will discuss the related results in
 Section~\ref{section_horizontal_heisenberg}.

                                 %%%%%%%%%%%%%% Coproduct and bialgebra pairing %%%%%%%%%%%%%%

\subsection{Hopf pairing, Drinfeld double and a universal $R$-matrix}
$\ $

   We recall the general notion of a Hopf pairing, following~\cite[Chapter 3]{KRT}.
  Given two Hopf algebras $A$ and $B$ with invertible antipodes $S_A$ and $S_B$,
  the bilinear map
    $$\varphi:A \times B\to \CC$$
  is called a \emph{Hopf pairing} if it satisfies the following properties:
\begin{equation*}
  \varphi(a,bb')=\varphi(a_1,b)\varphi(a_2,b')\ \ \ \forall\ a\in A,\ b,b'\in B ,
\end{equation*}
\begin{equation*}
  \varphi(aa',b)=\varphi(a,b_2)\varphi(a',b_1)\ \ \ \forall\ a,a'\in A,\ b\in B,
\end{equation*}
\begin{equation*}
  \varphi(a,1_B)=\epsilon_A(a)\ \mathrm{and}\  \varphi(1_A,b)=\epsilon_B(b)\ \ \ \forall\ a\in A,\ b\in B,
\end{equation*}
\begin{equation*}
  \varphi(S_A(a),b)=\varphi(a,S_B^{-1}(b))\ \ \ \forall\ a\in A,\ b\in B,
\end{equation*}
 where we use the Sweedler notation for the coproduct:
  $$\Delta(x)=x_1\otimes x_2.$$

 For such a data, one can define the \emph{generalized Drinfeld double} $D_\varphi(A,B)$ as follows:

\begin{thm}\cite[Theorem 3.2]{KRT}\label{generalized Drinfeld double}
    There is a unique Hopf algebra $D_\varphi(A,B)$ satisfying the following properties:

\medskip
 (i) As coalgebras $D_\varphi(A,B)\simeq A\otimes B$.

\medskip
 (ii) Under the natural inclusions
       $$A\hookrightarrow D_\varphi(A,B)\ \mathrm{given\ by}\ a\mapsto a\otimes 1_B,$$
       $$B\hookrightarrow D_\varphi(A,B)\ \mathrm{given\ by}\ b\mapsto 1_A\otimes b,$$
      $A$ and $B$ are Hopf subalgebras of $D_\varphi(A,B)$.

\medskip
  (iii) For any $a\in A, b\in B$, we have
         $$(a\otimes 1_B)\cdot (1_A\otimes b)=a\otimes b$$
        and
         $$(1_A\otimes b)\cdot (a\otimes 1_B)=\varphi(S_A^{-1}(a_1),b_1)\varphi(a_3,b_3)a_2\otimes b_2.$$
\end{thm}

\begin{rem}
 The notion of the \emph{Drinfeld double} is reserved for the case $B=A^{\star,\mathrm{cop}}$ with $\varphi$ being the natural pairing.
\end{rem}

   A Hopf algebra $A$ is \emph{quasitriangular} (\emph{formally quasitriangular}) if there is an invertible element
    $$R\in A\otimes A\ (\mathrm{or}\ R\in A\widehat{\otimes}A)$$
  satisfying the following properties:
  $$R\Delta(x)=\Delta^{\mathrm{op}}(x)R \ \ \forall\ x\in A,$$
  $$(\Delta\otimes \mathrm{Id})(R)=R^{13}R^{23},$$
  $$(\mathrm{Id}\otimes \Delta)(R)=R^{13}R^{12}.$$
  Such an element $R$ is called a \emph{universal $R$-matrix} of $A$.

  The fundamental property of Drinfeld doubles is their quasitriangularity:

\begin{thm}\cite[Theorem 3.2]{KRT}
  For a nondegenerate Hopf pairing $\varphi:A\times B\to \CC$, the
 generalized Drinfeld double $D_\varphi(A,B)$ is formally quasitriangular with
 the universal $R$-matrix
  $$R=\sum_{i} e_i\otimes e_i^*,$$
 where $\{e_i\}$ is a basis of $A$ and $\{e_i^*\}$ is the dual basis of $B$ (with respect to $\varphi$).
\end{thm}

                                 %%%%%%%%%%%%%% Coproduct and bialgebra pairing for toroidal algebra %%%%%%%%%%%%%%

\subsection{Quantum toroidal algebra $\ddot{U}_{q,d}(\ssl_n)$ as a Drinfeld double}
$\ $

  In order to apply the constructions of the previous section to the
 quantum toroidal algebra $\ddot{U}_{q,d}(\ssl_n)$ and its subalgebras, we need
 to endow the former with a Hopf algebra structure.
  This was first done (in a more general setup) in~\cite[Theorem 2.1]{DI}:

\begin{thm}\label{Hopf algebra}
  The formulas \emph{(H1-H9)} endow $\ddot{U}_{q,d}(\ssl_n)$ with a topological Hopf algebra structure:
\begin{equation}\tag{H1}\label{H1}
  \Delta(e_i(z))=e_i(z)\otimes 1+ \psi^-_i(\gamma^{1/2}_{(1)}z)\otimes e_i(\gamma_{(1)}z),
\end{equation}
\begin{equation}\tag{H2}\label{H2}
  \Delta(f_i(z))=1\otimes f_i(z)+f_i(\gamma_{(2)}z)\otimes \psi_i^+(\gamma^{1/2}_{(2)}z),
\end{equation}
\begin{equation}\tag{H3}\label{H3}
  \Delta(\psi^\pm_i(z))=\psi^\pm_i(\gamma^{\pm 1/2}_{(2)}z)\otimes \psi^\pm_i(\gamma^{\mp 1/2}_{(1)}z),
\end{equation}
\begin{equation}\tag{H4}\label{H4}
  \Delta(x)=x\otimes x\ \mathrm{for}\ x=\gamma^{\pm 1/2}, q^{\pm d_1}, q^{\pm d_2},
\end{equation}
\begin{equation}\tag{H5}\label{H5}
  \epsilon(e_i(z))=\epsilon(f_i(z))=0,\ \epsilon(\psi^{\pm}_i(z))=1,
\end{equation}
\begin{equation}\tag{H6}\label{H6}
  \epsilon(x)=1\ \mathrm{for}\ x=\gamma^{\pm 1/2}, q^{\pm d_1}, q^{\pm d_2},
\end{equation}
\begin{equation}\tag{H7}\label{H7}
  S(e_i(z))=-\psi^-_i(\gamma^{-1/2}z)^{-1}e_i(\gamma^{-1}z),
\end{equation}
\begin{equation}\tag{H8}\label{H8}
  S(f_i(z))=-f_i(\gamma^{-1}z)\psi^+_i(\gamma^{-1/2}z)^{-1},
\end{equation}
\begin{equation}\tag{H9}\label{H9}
  S(x)=x^{-1}\ \mathrm{for}\ x=\gamma^{\pm 1/2}, q^{\pm d_1}, q^{\pm d_2},\psi^\pm_i(z),
\end{equation}
   where $\gamma^{1/2}_{(1)}:=\gamma^{1/2}\otimes 1$ and $\gamma^{1/2}_{(2)}:=1\otimes \gamma^{1/2}$.
\end{thm}

   Let $\ddot{U}^{\geq}$ be the subalgebra of $\ddot{U}_{q,d}(\ssl_n)$ generated by
  $\{e_{i,k},\psi_{i,l}, \psi_{i,0}^{\pm 1}, \gamma^{\pm 1/2}, q^{\pm d_1}, q^{\pm d_2}\}_{k\in \ZZ}^{l\in -\NN}$, and
  let $\ddot{U}^{\leq}$ be the subalgebra of $\ddot{U}_{q,d}(\ssl_n)$ generated by
  $\{f_{i,k},\psi_{i,l}, \psi_{i,0}^{\pm 1}, \gamma^{\pm 1/2}, q^{\pm d_1}, q^{\pm d_2}\}_{k\in \ZZ}^{l\in \NN}$.
  Now we are ready to state the main result of this section (the proof is straightforward):

\begin{thm}\label{Drinfeld double sln}
 (a)  There exists a unique Hopf algebra  pairing $\varphi: \ddot{U}^\geq \times \ddot{U}^\leq\to \CC$  satisfying
\begin{equation}\tag{P1}\label{P1}
     \varphi(e_i(z), f_j(w))=\frac{\delta_{i,j}}{q-q^{-1}}\cdot \delta\left(\frac{z}{w}\right),\
     \varphi(\psi^-_i(z),\psi^+_j(w))=g_{a_{i,j}}(d^{m_{i,j}}z/w),
\end{equation}
\begin{equation}\tag{P2}\label{P2}
     \varphi(e_i(z),x^-)=\varphi(x^+,f_i(z))=0\ \mathrm{for}\  x^\pm=\psi^\mp_j(w),\psi_{j,0}^{\pm 1}, \gamma^{1/2}, q^{d_1}, q^{d_2},
\end{equation}
\begin{equation}\tag{P3}\label{P3}
     \varphi(\gamma^{1/2}, q^{d_1})=\varphi(q^{d_1},\gamma^{1/2})=q^{-1/2},\
     \varphi(\psi^-_i(z),q^{d_2})=q^{-1},\ \varphi(q^{d_2},\psi^+_i(z))=q,
\end{equation}
\begin{equation}\tag{P4}\label{P4}
     \varphi(\psi^-_i(z),x)=\varphi(x,\psi^+_i(z))=1\ \mathrm{for}\ x=\gamma^{1/2}, q^{d_1},
\end{equation}
\begin{equation}\tag{P5}\label{P5}
     \varphi(\gamma^{1/2}, q^{d_2})=\varphi(q^{d_2},\gamma^{1/2})=\varphi(q^{d_a},q^{d_b})=\varphi(\gamma^{1/2},\gamma^{1/2})=1\ \mathrm{for}\ 1\leq a,b\leq 2.
\end{equation}

\medskip
\noindent
 (b) The natural Hopf algebra homomorphism
       $D_\varphi(\ddot{U}^\geq, \ddot{U}^\leq)\to\ddot{U}_{q,d}(\ssl_n)$
     induces the isomorphism
       $$\Xi:D_\varphi(\ddot{U}^\geq, \ddot{U}^\leq)/I\iso \ddot{U}_{q,d}(\ssl_n)\ \mathrm{with}\
         I:=(x\otimes 1-1\otimes x|x=\psi_{i,0}^{\pm 1}, \gamma^{\pm 1/2}, q^{\pm d_1}, q^{\pm d_2}).$$

\noindent
 (c) Consider a slight modification $\ddot{U}^{'}_{q,d}(\ssl_n)$, obtained from $\ddot{U}_{q,d}(\ssl_n)$
   by ``throwing away'' the generator $q^{\pm d_2}$ and taking the quotient by the central element $c'$.
   As in (b), this algebra admits the double Drinfeld realization via $D_{\varphi'}(\ddot{U}^{'\geq},\ddot{U}^{'\leq})$,
   where $\ddot{U}^{'\leq}$ and $\ddot{U}^{'\geq}$ are obtained from $\ddot{U}^{\leq}$ and $\ddot{U}^{\geq}$ by ``throwing away'' $q^{\pm d_2}$
   and taking the quotient by $c'$, while $\varphi'$ is induced by $\varphi$.

\medskip
\noindent
 (d)  The pairings $\varphi$ and $\varphi'$ are nondegenerate if and only if $q,qd,qd^{-1}$ are not roots of unity.
\end{thm}

                                 %%%%%%%%%%%%%% Transfer matrix %%%%%%%%%%%%%%

\subsection{Bethe subalgebras}\label{section_bethe}
$\ $

  Let us recall the standard way of constructing large commutative
 subalgebras of a (formally) quasitriangular Hopf algebra $A$.
 Fix a group-like element $x\in A$ (or in an appropriate completion $x\in A^\wedge$).
 For an $A$-representation $V$, we consider the \emph{transfer matrix}
    $$T_V(x):=(1\otimes \tr_V)((1\otimes x)R)$$
 if the latter is well defined. The properties of the $R$-matrix imply
   $$T_{V_1\oplus V_2}(x)=T_{V_1}(x)+T_{V_2}(x),$$
   $$T_{V_1\otimes V_2}(x)=T_{V_2}(x)\cdot T_{V_1}(x).$$
  In particular, we see that $T_{V_1}(x)\cdot T_{V_2}(x)=T_{V_2}(x)\cdot T_{V_1}(x)$.

 To summarize, $\bullet\mapsto T_{\bullet}(x)$ is a ring homomorphism from the Grothendieck group
 of any suitable tensor category of $A$-modules to the suitable completion $A^\wedge$, with
 the image being a commutative subalgebra of that completion.
  The commutative subalgebras constructed in this way are sometimes called the \emph{Bethe (sub)algebras}.

                     %%%%%%%%%%%%%%%%%%%%%%%%%%%%%%%%%%%%%%%%%%%%%
                     %%%%%%%%%%%%%%%%%% REMARK #3 %%%%%%%%%%%%%%%%
                     %%%%%%%%%%%%%%%%%%%%%%%%%%%%%%%%%%%%%%%%%%%%%
  %  As proposed by B.~Enriquez, it might be nice to explain what we mean by a
  % "suitable tensor category" and "suitable completion".
  %  However, we feel lazy to give a general framework for that, though the reader
  % can get an idea just by looking at the case of quantum affine/loop algebras.
  %  Instead, we just specify below the two particular cases which are of interest for us.
                     %%%%%%%%%%%%%%%%%%%%%%%%%%%%%%%%%%%%%%%%%%%%%
                     %%%%%%%%%%%%%%%%%%% END %%%%%%%%%%%%%%%%%%%%%
                     %%%%%%%%%%%%%%%%%%%%%%%%%%%%%%%%%%%%%%%%%%%%%

   In Section 3, we will apply this construction to the following two cases:

 $\circ$ The formally quasitriangular algebra is $A=\ddot{U}^{'}_{q,d}(\ssl_n)$, the corresponding
 group-like element is $x=q^{\lambda_1 h_{1,0}+\cdots+\lambda_{n-1}h_{n-1,0}+\lambda_n d_1}$, and we consider
 a tensor category of $\ddot{U}^{'}_{q,d}(\ssl_n)$-representations generated by vertex $\ddot{U}^{'}_{q,d}(\ssl_n)$-representations
 $\rho_{p,\bar{c}}$ from Section 3.1.\footnote{\ Actually, one can consider the whole category of highest weight $\ddot{U}^{'}_{q,d}(\ssl_n)$-representations,
 see~\cite{M}.}

 $\circ$ The formally quasitriangular algebra is $A=U_q(L\gl_n)$ (see Section 3.4), the corresponding
 group-like element is $x=q^{\lambda_1 h_{1,0}+\cdots+\lambda_{n-1}h_{n-1,0}}$
 (the most generic element of the \emph{finite Cartan} part), and we consider the tensor category of all finite-dimensional
 $U_q(L\gl_n)$-representations.

                                 %%%%%%%%%%%%%% Small Shuffle algebra %%%%%%%%%%%%%%

\subsection{Small shuffle algebra}\label{section_small_shuffle}
 $\ $

  As a motivating point for the current paper, we briefly recall the
 notion of the \emph{small} shuffle algebra and its particular commutative subalgebra.
 Let $\ZZ_+:=\{n\in\ZZ|n\geq 0\}=\NN\cup \{0\}$.
  Consider a $\ZZ_+$-graded $\CC$-vector space $\sS^{\mathrm{sm}}=\bigoplus_{n\geq 0}\sS^{\mathrm{sm}}_n$,
 where $\sS^{\mathrm{sm}}_n$ consists of rational functions $\frac{f(x_1,\ldots,x_n)}{\Delta(x_1,\ldots,x_n)}$
 with $f\in \CC[x_1^{\pm 1},\ldots,x_n^{\pm 1}]^{\mathfrak{S}_n}$ and $\Delta(x_1,\ldots,x_n):=\prod_{i\ne j}(x_i-x_j)$.
  Define the star product $\overset{\mathrm{sm}}\star:\sS^{\mathrm{sm}}_k\times \sS^{\mathrm{sm}}_l\to \sS^{\mathrm{sm}}_{k+l}$
  by
 $$(F\overset{\mathrm{sm}}\star G)(x_1,\ldots,x_{k+l}):=\mathrm{Sym}_{\mathfrak{S}_{k+l}}\left(F(x_1,\ldots,x_k)G(x_{k+1},\ldots,x_{k+l})
   \prod_{i\leq k}^{j>k}\lambda(x_i/x_j)\right)$$
 with
  $$\lambda(x):=\frac{(q_1x-1)(q_2x-1)(q_3x-1)}{(x-1)^3},\ \mathrm{where}\ q_i\in \CC\backslash\{0,1\}\ \mathrm{and}\ q_1q_2q_3=1.$$
 This endows $\sS^{\mathrm{sm}}$ with a structure of an associative unital $\CC$-algebra with the unit $\textbf{1}\in \sS^{\mathrm{sm}}_0$.

  We say that an element $\frac{f(x_1,\ldots,x_n)}{\Delta(x_1,\ldots,x_n)}\in \sS^{\mathrm{sm}}_n$
 satisfies the \emph{wheel conditions} if and only if
  $$f(x_1,\ldots,x_n)=0\ \mathrm{once}\ x_{i_1}/x_{i_2}=q_1\ \mathrm{and}\ x_{i_2}/x_{i_3}=q_2\ \mathrm{for\ some}\ 1\leq i_1,i_2,i_3\leq n.$$
  Let $S^{\mathrm{sm}}\subset \sS^{\mathrm{sm}}$ be a $\ZZ_+$-graded subspace, consisting of all such elements.
 The subspace $S^{\mathrm{sm}}$ is $\overset{\mathrm{sm}}\star$-closed (see~\cite[Proposition 2.10]{FS}).

\begin{defn}
  The algebra $(S^\mathrm{sm},\overset{\mathrm{sm}}\star)$ is called the \textit{small shuffle algebra}.
\end{defn}

   Following~\cite{FS}, we introduce an important $\ZZ_+$-graded subspace $\A^{\mathrm{sm}}=\bigoplus_{n} \A^{\mathrm{sm}}_n$ of $S^{\mathrm{sm}}$.
  Its degree $n$ component is defined by
   $$\A^{\mathrm{sm}}_n:=\{F\in S^{\mathrm{sm}}_n|\partial^{(0;k)}F, \partial^{(\infty;k)}F\ \mathrm{exist\ and\ }
     \partial^{(0;k)}F=\partial^{(\infty;k)}F\ \ \forall\ 0\leq k\leq n\},$$
  where
  $$\partial^{(0;k)}F:=\underset{\xi\to 0}\lim F(x_1,\ldots,\xi\cdot x_{n-k+1},\ldots,\xi\cdot x_n),\
    \partial^{(\infty;k)}F:=\underset{\xi\to \infty}\lim F(x_1,\ldots,\xi\cdot x_{n-k+1},\ldots,\xi\cdot x_n)$$
  whenever the limits exist.

  This subspace satisfies the following properties:

\begin{thm}\cite[Section 2]{FS}\label{comm_subalg_m}
 We have:

\noindent
 (a) Suppose $F\in S^{\mathrm{sm}}_n$ and $\partial^{(\infty;k)}F$ exist for all $0\leq k\leq n$, then $F\in \A^{\mathrm{sm}}_n$.

\noindent
 (b) The subspace $\A^{\mathrm{sm}}\subset S^{\mathrm{sm}}$ is $\overset{\mathrm{sm}}\star$-commutative.

\noindent
 (c) $\A^{\mathrm{sm}}$ is $\overset{\mathrm{sm}}\star$-closed and it is a
 polynomial algebra in $\{K_j\}_{j\geq 1}$ with $K_j\in S^{\mathrm{sm}}_j$ defined by:
  $$K_1(x_1)=x_1^0,\ K_2(x_1,x_2)=\frac{(x_1-q_1x_2)(x_2-q_1x_1)}{(x_1-x_2)^2},\ K_m(x_1,\ldots,x_m)=\prod_{1\leq i<j\leq m} K_2(x_i,x_j).$$
\end{thm}

                                 %%%%%%%%%%%%%% Big Shuffle algebra %%%%%%%%%%%%%%

\subsection{Big shuffle algebra}
$\ $

  Consider a $\ZZ_+^{[n]}$-graded $\CC$-vector space
    $$\sS=\underset{\overline{k}=(k_0,\ldots,k_{n-1})\in \ZZ_+^{[n]}}\bigoplus\sS_{\overline{k}},$$
 where $\sS_{k_0,\ldots,k_{n-1}}$ consists of $\prod\mathfrak{S}_{k_i}$-symmetric rational functions in the variables
  $\{x_{i,j}\}_{i\in [n]}^{1\leq j\leq k_i}$.
 We also fix an $n\times n$ matrix of rational functions
   $\Omega=(\omega_{i,j}(z))_{i,j\in [n]} \in \mathrm{Mat}_{n\times n}(\CC(z))$
 by setting
  $$\omega_{i,i}(z)=\frac{z-q^{-2}}{z-1},\ \omega_{i,i+1}(z)=\frac{d^{-1}z-q}{z-1},\ \omega_{i,i-1}(z)=\frac{z-qd^{-1}}{z-1},\
    \mathrm{and}\ \omega_{i,j}(z)=1\ \mathrm{otherwise}.$$
  Let us now introduce the bilinear $\star$ product on $\sS$: given  $f\in \sS_{\overline{k}}, g\in \sS_{\overline{l}}$ define
 $f\star g\in \sS_{\overline{k}+\overline{l}}$ by
 $$(f\star g)(x_{0,1},\ldots,x_{0,k_0+l_0};\ldots;x_{n-1,1},\ldots, x_{n-1,k_{n-1}+l_{n-1}}):=$$
 $$\Sym_{\prod\mathfrak{S}_{k_i+l_i}}
   \left(f(\{x_{i,j}\}_{i\in [n]}^{1\leq j\leq k_i})
   g(\{x_{i,j}\}_{i\in [n]}^{k_i<j\leq k_i+l_i})\times
   \prod_{i\in [n]}^{i'\in [n]}\prod_{j\leq k_i}^{j'>k_{i'}}\omega_{i,i'}(x_{i,j}/x_{i',j'})\right).$$

\medskip
  This endows $\sS$ with a structure of an associative unital algebra with the unit $\textbf{1}\in \sS_{0,\ldots,0}$.
 We will be interested only in a certain subspace of $\sS$, defined by the \emph{pole} and \emph{wheel conditions}:

\medskip
 \noindent
 $\bullet$ We say that $F\in \sS_{\overline{k}}$ satisfies the \emph{pole conditions} if and only if
 $$F=\frac{f(x_{0,1},\ldots,x_{n-1,k_{n-1}})}{\prod_{i\in [n]}\prod_{j\leq k_i}^{j'\leq k_{i+1}}(x_{i,j}-x_{i+1,j'})},\
   \mathrm{where}\ f\in (\CC[x_{i,j}^{\pm 1}]_{i\in [n]}^{1\leq j\leq k_i})^{\prod \mathfrak{S}_{k_i}}.$$

 \noindent
 $\bullet$
  We say that $F\in \sS_{\overline{k}}$ satisfies the \emph{wheel conditions} if and only if
   $$F(x_{0,1},\ldots,x_{n-1,k_{n-1}})=0\
     \mathrm{once}\ x_{i,j_1}/x_{i+\epsilon,l}=qd^{\epsilon}\ \mathrm{and}\ x_{i+\epsilon,l}/x_{i,j_2}=qd^{-\epsilon}\ \mathrm{for\ some}\ i,\epsilon,j_1,j_2,l,$$
 where $\epsilon\in \{\pm 1\},\ i\in [n],\ 1\leq j_1,j_2\leq k_i,\ 1\leq l\leq k_{i+\epsilon}$
 and we use the cyclic notation $x_{n,l}:=x_{0,l},\ k_n:=k_0,\ x_{-1,l}:=x_{n-1,l},\ k_{-1}:=k_{n-1}$ as before.

\medskip
 \noindent
  Let $S_{\overline{k}}\subset \sS_{\overline{k}}$ be the subspace of all elements $F$ satisfying the above two conditions and set
 $$S:=\underset{\overline{k}\in \ZZ_+^{[n]}}\oplus S_{\overline{k}}.$$
  Further $S_{\overline{k}}=\oplus_{d\in \ZZ}S_{\overline{k},d}$ with
  $S_{\overline{k},d}:=\{F\in S_{\overline{k}}|\mathrm{tot.deg}(F)=d\}$.
 The following is straightforward:

\begin{lem}
 The subspace $S\subset\sS$ is $\star$-closed.
\end{lem}

 Now we are ready to introduce the main algebra of this paper:

\begin{defn}
 The algebra $(S,\star)$ is called the big shuffle algebra (of $A_{n-1}^{(1)}$-type).
\end{defn}

\subsection{Relation between $S$ and $\ddot{U}^+$}
$\ $

  Recall the subalgebra $\ddot{U}^+$ of $\ddot{U}_{q,d}(\ssl_n)$ from Section 1.1. By standard results, $\ddot{U}^+$
 is generated by $\{e_{i,k}\}_{i\in [n]}^{k\in \ZZ}$ with the defining relations (T2, T7.1). The following is straightforward:

\begin{prop}\label{homomorphisms}
  There exists a unique algebra homomorphism $\Psi:\ddot{U}^+\to \sS$ such
 that $\Psi(e_{i,k})=x_{i,1}^k\ \forall\ i\in [n], k\in \ZZ$.
\end{prop}

  As a consequence, $\mathrm{Im}(\Psi)\subset S$.
 The following beautiful result was recently proved by Negut:

\begin{thm}\cite[Theorem 1.1]{N2}\label{Negut theorem}
 The homomorphism $\Psi:\ddot{U}^+\to S$ is an isomorphism of $\ZZ_+^{[n]}\times \ZZ$-graded algebras.
\end{thm}

\begin{rem}
  In the \emph{loc. cit.} $d=1$, but the proof can be easily modified for any $d$.
  Note that the algebra $\A^+$ from~\cite{N2} is isomorphic to our $S$ with $d=1$ via the  map $S_{\mid_{d=1}}\to \A^+$ given by
  $$F(\{x_{i,j}\}_{i\in [n]}^{1\leq j\leq k_i})\mapsto q^{\sum_{i=0}^{n-1} \frac{k_i(k_i-1)}{2}}F(\{z_{i,j}\}_{1\leq i\leq n}^{1\leq j\leq k_{i-1}})\cdot
    \prod_{i=1}^n\prod_{j\ne j'}\frac{z_{i,j}-z_{i,j'}}{q^{-1}z_{i,j}-qz_{i,j'}}\cdot \prod_{i=1}^n\prod _{j,j'}\frac{z_{i,j}-z_{i+1,j'}}{z_{i,j}-qz_{i+1,j'}}.$$
\end{rem}

%%%%%%%%%%%%%%%%%%%%%%%%%%%%%%%%%%%%%%% SECTION 2 %%%%%%%%%%%%%%%%%%%%%%%%%%%%%%%%%%%%%%%%%%%%%%%%%%%%%%%

\section{Subalgebras $\A(s_0,\ldots,s_{n-1})$}

    %%%%%%%%%% Limits at 0 and infinity, definition of A(s_1,...,s_n), elements F_k,mu%%%%%%%%%%%%

\subsection{Key constructions}
$\  $

  In this section, we introduce the key objects of our paper, the commutative subalgebras of $S$, analogous to $\A^{\mathrm{sm}}\subset S^{\mathrm{sm}}$
 from Section~\ref{section_small_shuffle}.
 The new feature of our setup (in comparison to the small shuffle algebras) is that we get an $(n-1)$-parameter family of those.

  For any $0\leq \overline{l}\leq \overline{k}\in \ZZ_+^{[n]}$, $\xi\in \CC^*$ and $F\in S_{\overline{k}}$,
 we define $F^{\overline{l}}_\xi\in \CC(x_{0,1},\ldots,x_{n-1,k_{n-1}})$ by
  $$F^{\overline{l}}_\xi:=
    F(\xi\cdot x_{0,1},\ldots,\xi\cdot x_{0,l_0},x_{0,l_0+1},\ldots,x_{0,k_0};
      \ldots;\xi\cdot x_{n-1,1},\ldots,\xi \cdot x_{n-1,l_{n-1}}, x_{n-1,l_{n-1}+1},\ldots).$$
 For any integer numbers $a\leq b$, define the degree vector $\overline{l}:=[a;b]\in \ZZ_+^{[n]}$ by
  $$\overline{l}=(l_0,\ldots,l_{n-1}) \ \ \mathrm{with}\ \ l_i=\#\{c\in \ZZ|a\leq c\leq b\ \mathrm{and}\ c\equiv i\ (\mathrm{mod}\ n)\}.$$
 For such a choice of $\overline{l}$, we will denote $F^{\overline{l}}_\xi$ simply by $F^{(a,b)}_\xi$.

\begin{defn}
   For any $\overline{s}=(s_0,\ldots,s_{n-1})\in (\CC^*)^{[n]}$, consider a $\ZZ_+^{[n]}$-graded subspace
 $\A(\overline{s})\subset S$ whose degree $\overline{k}=(k_0,\ldots,k_{n-1})$
 component is defined by
   $$\Aa(\overline{s})_{\overline{k}}:=
     \left\{F\in S_{\overline{k},0}\mid \partial^{(\infty;a,b)}F=\prod_{i=a}^b s_i \cdot \partial^{(0;a,b)}F
     \ \ \forall\ a,b\in \ZZ\ \mathrm{such\ that}\ a\leq b \ \mathrm{and}\ [a;b]\leq \overline{k}\right\},$$
 where $\partial^{(\infty;a,b)}F:=\underset{\xi\to \infty}\lim F^{(a,b)}_\xi, \partial^{(0;a,b)}F:=\underset{\xi\to 0}\lim F^{(a,b)}_\xi$
 whenever these limits exist, $s_i:=s_{i\ \mathrm{mod}\ n}$.
\end{defn}

  A certain class of such elements is provided by the following result:

\begin{lem}\label{lemma 0}
  For any $k\in \NN,\ \mu\in \CC$, and $\overline{s}\in (\CC^*)^{[n]}$, define $F_k^{\mu}(\overline{s})\in  S_{k,\ldots,k}$ by
 $$F_k^{\mu}(\overline{s}):=
   \frac{\prod_{i\in [n]}\prod_{1\leq j\ne j'\leq k}(x_{i,j}-q^{-2}x_{i,j'})\cdot \prod_{i\in [n]}(s_0\ldots s_i \prod_{j=1}^k x_{i,j}-\mu\prod_{j=1}^k x_{i+1,j})}
   {\prod_{i\in [n]} \prod_{1\leq j,j'\leq k}(x_{i,j}-x_{i+1,j'})},$$
 where we set $x_{n,j}:=x_{0,j}$ as before.
 If $s_0\ldots s_{n-1}=1$, then $F_k^{\mu}(\overline{s})\in \A(\overline{s})$.
\end{lem}

\begin{proof}
 $\ $

  Without loss of generality, we can assume $\mu\ne 0, a=0, b=nr+c,\ 0\leq r\leq k-1, 0\leq c\leq n-1$.
 Then $l_0=\ldots=l_c=r+1$ and $l_{c+1}=\ldots=l_{n-1}=r$.
  As $\xi\to \infty$, the function $F_k^{\mu}(\overline{s})_\xi^{(a,b)}$ grows at
 the  speed $\xi^{\sum_{i\in [n]} l_i(l_{i+1}-l_i-1)+\sum_{i\in [n]}\mathrm{max}\{l_i,l_{i+1}\}}$,
 while as $\xi\to 0$, the function $F_k^{\mu}(\overline{s})_\xi^{(a,b)}$ grows at
 the  speed $\xi^{\sum_{i\in [n]} l_i(-l_{i+1}+l_i-1)+\sum_{i\in [n]} \mathrm{min}\{l_i,l_{i+1}\}}$.
  For the above values of $l_i$, both powers of $\xi$ are zero and
 hence both limits $\partial^{(\infty;a,b)}F_k^{\mu}(\overline{s})$ and $\partial^{(0;a,b)}F_k^{\mu}(\overline{s})$
 exist. Moreover, for $\alpha$ being $0$ or $\infty$, we have
 $\partial^{(\alpha;a,b)}F_k^{\mu}(\overline{s})=(-1)^{\sum_{i\in [n]} l_i(l_i-l_{i-1})}q^{-2\sum_{i\in [n]} l_i(k-l_i)}\cdot G\cdot \prod_{i\in [n]} G_{\alpha,i}$,
 where
 $$G=\frac{\prod_{i\in [n]}\prod_{1\leq j\ne j'\leq l_i} (x_{i,j}-q^{-2}x_{i,j'})\cdot\prod_{i\in [n]}\prod_{l_i< j\ne j'\leq k} (x_{i,j}-q^{-2}x_{i,j'})}
            {\prod_{i\in [n]}\prod_{1\leq j\leq l_i}^{1\leq j'\leq l_{i+1}}
            (x_{i,j}-x_{i+1,j'})\cdot \prod_{i\in [n]}\prod_{l_i< j \leq k}^{l_{i+1}<j'\leq k } (x_{i,j}-x_{i+1,j'})},$$
\[
  G_{\infty,i}=\prod_{j=1}^{l_i} x_{i,j}^{l_{i+1}+l_{i-1}-2l_i}\cdot
   \begin{cases}
    s_0\ldots s_i \prod_{j=1}^k x_{i,j}-\mu\prod_{j=1}^k x_{i+1,j} & \text{if } l_i=l_{i+1} \\
    s_0\ldots s_i \prod_{j=1}^k x_{i,j}   & \text{if } l_i>l_{i+1} \\
    -\mu\prod_{j=1}^k x_{i+1,j} & \text{if } l_i<l_{i+1}
   \end{cases},
\]
\[
  G_{0,i}=\prod_{j=l_i+1}^{k} x_{i,j}^{-l_{i+1}-l_{i-1}+2l_i}\cdot
   \begin{cases}
    s_0\ldots s_i \prod_{j=1}^k x_{i,j}-\mu\prod_{j=1}^k x_{i+1,j} & \text{if } l_i=l_{i+1} \\
    -\mu\prod_{j=1}^k x_{i+1,j} & \text{if } l_i>l_{i+1} \\
    s_0\ldots s_i \prod_{j=1}^k x_{i,j} & \text{if } l_i<l_{i+1}
   \end{cases}.
\]
  The equality $\partial^{(\infty;a,b)}F_k^{\mu}(\overline{s})=\prod_{i=0}^c s_i \cdot \partial^{(0;a,b)}F_k^{\mu}(\overline{s})$ follows,
 while $\mathrm{tot.deg}\ F_k^{\mu}(\overline{s})=0$.
\end{proof}

       %%%%%%%%%%%%%%%%%% The main result: description of A(s_1,...,s_n) for generic s_i s.t. s_1...s_n=1 %%%%%%%%%%%%%%

\subsection{Main result}
$\ $

  A collection $\{s_0,\ldots,s_{n-1}\}\subset \CC^*$ satisfying $s_0\ldots s_{n-1}=1$ is called \emph{generic} if and only if
 $$s_0^{\alpha_0}\ldots s_{n-1}^{\alpha_{n-1}}\in q^\ZZ \cdot d^\ZZ \Rightarrow \alpha_0=\ldots=\alpha_{n-1}.$$
 The main result of this section describes $\A(s_0,\ldots,s_{n-1})$ for such generic $n$-tuples $\{s_0,\ldots,s_{n-1}\}$.

\begin{thm}\label{main1}
  For a generic $\overline{s}=(s_0,\ldots,s_{n-1})$ satisfying $s_0\ldots s_{n-1}=1$, the
 space $\A(\overline{s})$ is shuffle-generated by $\{F_k^{\mu}(\overline{s})|k\in \NN, \mu\in \CC\}$.
 Moreover, $\A(\overline{s})$ is a polynomial algebra in free generators $\{F_k^{\mu_l}(\overline{s})|k\in \NN, 1\leq l\leq n\}$
 for arbitrary pairwise distinct $\mu_1,\ldots,\mu_n\in \CC$.
 In particular, $\A(\overline{s})$ is a commutative subalgebra of $S$.
\end{thm}

 The proof of this theorem will proceed in several steps.
  First, we will use an analogue of the \emph{Gordon filtration}
 from~\cite{FS}, further generalized in~\cite{N2} to prove Theorem~\ref{Negut theorem}, in order
 to obtain the upper bound on dimensions of $\A(\overline{s})_{\overline{k}}$.
  Next, we will show that the subalgebra $\A'(\overline{s})\subset S$, shuffle generated by all
 $F_k^{\mu}(\overline{s})$, belongs to $\A(\overline{s})$. We will use another
 filtration to argue that the dimension of $\A'(\overline{s})_{\overline{k}}$
 is at least as big as the upper bound for the dimension of
 $\A(\overline{s})_{\overline{k}}$, implying $\A'(\overline{s})=\A(\overline{s})$.
  Similar arguments will also imply the commutativity of $\A(\overline{s})$.

\begin{lem}\label{lemma 1}
  Consider the polynomial algebra $\mathcal{R}=\CC[T_{i,m}]^{m\geq 1}_{i\in [n]}$ with $\deg(T_{i,m})=m$. Then:

\noindent
 (a) For $\overline{k}=k\delta:=(k,\ldots,k)$, we have
     $\dim \A(\overline{s})_{\overline{k}}\leq \dim \mathcal{R}_{k}$.

\noindent
 (b) For $\overline{k}\notin\{0,\delta,2\delta,\ldots\}$, we have $\A(\overline{s})_{\overline{k}}=0$.
\end{lem}

\begin{proof}
 $\ $

   An unordered set $L$ of integer intervals $\{[a_1,b_1],\ldots, [a_r,b_r]\}$ is
 called a \emph{partition} of $\overline{k}\in \ZZ_+^{[n]}$ (denoted by $L\vdash \overline{k}$) if $\overline{k}=[a_1;b_1]+\cdots+[a_r;b_r]$.
 We order the elements of $L$ so that $b_1-a_1\geq b_2-a_2\geq \cdots \geq b_r-a_r$.
 The two sets $L$ and $L'$ as above are said to be equivalent if $|L|=|L'|$, and we
 can order their elements so that $b_i'-b_i=a_i'-a_i=nc_i$ for all $i$ and some $c_i\in \ZZ$.
  Note that the collection of $L\vdash\overline{k}$,
 up to the above equivalence, is finite for any $\overline{k}\in \ZZ_+^{[n]}$.
  Finally, we say $L'>L$ if there exists $s$, such that
 $b'_s-a'_s>b_s-a_s$ and $b'_t-a'_t=b_t-a_t$ for $1\leq t\leq s-1$.

  Any $L\vdash \overline{k}$ defines a linear map $\phi_L: \A(\overline{s})_{\overline{k}}\to \CC[y_1^{\pm 1},\ldots, y_r^{\pm 1}]$
 as follows. Split the variables $\{x_{i,j}\}$ in $r$ groups, each group corresponding to
 one of the intervals in $L$. Specialize the
 variables corresponding to
 the interval $[a_t,b_t]$ to
 $(qd)^{-a_t}\cdot y_t,\ldots,(qd)^{-b_t}\cdot y_t$ in the natural order.
  For
   $$F=\frac{f(x_{0,1},\ldots,x_{n-1,k_{n-1}})}{\prod_{i\in [n]}\prod_{1\leq j\leq k_i}^{1\leq j'\leq k_{i+1}} (x_{i,j}-x_{i+1,j'})}\in
     \A(\overline{s})_{\overline{k}},$$
 define $\phi_L(F)$ as the corresponding specialization of $f$.
 The result is independent of our splitting of variables since $f$
 is symmetric.
  Finally, we define the filtration
 on $\A(\overline{s})_{\overline{k}}$ by
 $$\A(\overline{s})_{\overline{k}}^L:=\underset{L'>L}\bigcap \mathrm{Ker}(\phi_{L'}).$$

  Let us now consider the images $\phi_L(\A(\overline{s})^L_{\overline{k}})$ for any $L\vdash \overline{k}$.
 For $F\in \A(\overline{s})^L_{\overline{k}}$, we have:

\noindent
 $\circ$ The total degree $\mathrm{tot.deg}(\phi_L(F))=\sum_{i\in [n]} k_ik_{i+1}$, since
     $\mathrm{tot.deg} (F)=0$.

\noindent
 $\circ$ For each $1\leq t\leq r$, the degree of $\phi_L(F)$ with respect to $y_t$ is bounded by
     $$\deg_{y_t}(\phi_L(F))\leq \sum_{i\in [n]} (l^t_{i}(k_{i-1}+k_{i+1})-l^t_il^t_{i+1})$$
      due to the existence of the limit $\partial^{(\infty;a_t,b_t)}F$ (here $\overline{l}^t:=[a_t;b_t]\in \ZZ^{[n]}_+$ for $1\leq t\leq r$).

  On the other hand, the wheel conditions for $F$ guarantee that $\phi_L(F)(y_1,\ldots,y_r)$
 becomes zero under the following specializations:

 (i) $(qd)^{-x'}y_v=(q/d)(qd)^{-x}y_u$ for any $1\leq u<v\leq r,\ a_u\leq x<b_u, a_v\leq x'\leq b_v,\ x'\equiv x+1,$

 (ii) $(qd)^{-x'}y_v=(d/q)(qd)^{-x}y_u$ for any $1\leq u<v\leq r,\ a_u< x\leq b_u, a_v\leq x'\leq b_v,\ x'\equiv x-1.$

  Finally, the conditions $\phi_{L'}(F)=0$ for any $L'>L$ guarantee
 that $\phi_L(F)(y_1,\ldots,y_r)$ becomes zero under the following specializations:

 (iii) $(qd)^{-x'}y_v=(qd)^{-b_u-1}y_u$ for any $1\leq u<v\leq r,\ a_v\leq x'\leq b_v,\ x'\equiv b_u+1,$

 (iv) $(qd)^{-x'}y_v=(qd)^{-a_u+1}\cdot y_u$ for any $1\leq u<v\leq r,\ a_v\leq x'\leq b_v,\ x'\equiv a_u-1.$

  In particular, we see that
 $\phi_L(F)$ is divisible by $Q_L\in \CC[y_1,\ldots,y_r]$, defined as
 a product of the linear terms in $y_t$ coming from (i)--(iv)
 (if some of these coincide, we still count them with the correct multiplicity).
 Note that
  $$\mathrm{tot.deg}(Q_L)=\sum_{1\leq u<v\leq r}\sum_{i\in [n]} (l^u_i l^v_{i+1}+l^u_i l^v_{i-1})=
    \sum_{i\in [n]} k_ik_{i+1}-\sum_{t=1}^r\sum_{i\in [n]} l^t_i l^t_{i+1},$$
 while the degree with respect to each variable $y_t\ (1\leq t\leq r)$ is given by
  $$\deg_{y_t}(Q_L)=\sum_{i\in [n]} (l^t_i (k_{i-1}+k_{i+1})-2l^t_il^t_{i+1}).$$
 Define $r_L:=\phi_L(F)/Q_L\in \CC[y_1^{\pm 1},\ldots, y_r^{\pm 1}]$. Then:
  $$\mathrm{tot.deg}(r_L)=\sum_{t=1}^r\sum_{i\in [n]} l^t_i l^t_{i+1}\ \mathrm{and}\ \deg_{y_t}(r_L)\leq \sum_{i\in [n]} l^t_i l^t_{i+1}.$$
 Hence, $r_L=\nu\cdot \prod_{t=1}^r y_t^{\sum_{i\in [n]} l^t_i l^t_{i+1}}$ for some $\nu\in \CC$, so that
 $\phi_L(F)=\nu\cdot \prod_{t=1}^r y_t^{\sum_{i\in [n]} l^t_i l^t_{i+1}}\cdot Q_L$.
  On the other hand, applying the above specialization to the entire function $F$ rather than $f$, we get
 $\phi_L(F)/Q$, where $Q\in \CC[y_1,\ldots,y_r]$ is given by
  $$Q=\nu'\cdot \prod_{t=1}^r y_t^{\sum_{i\in [n]} l^t_i l^t_{i+1}}\cdot
    \prod_{1\leq u< v\leq r}\prod_{a_u\leq x\leq b_u}\prod_{a_v\leq x'\leq b_v}^{x'\equiv x\pm 1} ((qd)^{-x}y_u-(qd)^{-x'}y_v)\ \mathrm{for\ some}\ \nu'\in \CC^*.$$
 The condition $F\in \A(\overline{s})$ implies
  $$\underset{\xi\to \infty}\lim \left(\frac{\phi_L(F)}{Q}\right)_{\mid y_t\mapsto \xi \cdot y_t}=
    s_{a_t}\ldots s_{b_t}\cdot \underset{\xi\to 0}\lim \left(\frac{\phi_L(F)}{Q}\right)_{\mid y_t\mapsto \xi \cdot y_t} \ \ \forall\ 1\leq t\leq r.$$
 For $\nu\ne 0$, this equality enforces $s_{a_t}\ldots s_{b_t}\in q^\ZZ\cdot d^\ZZ$.
  Due to our condition on $\{s_i\}$, we get $b_t-a_t+1=nc_t$ for
 every $1\leq t \leq r$ and some $c_t\in \NN$.
  The claim (ii) of the lemma is now obvious, while
 part (i) of the lemma follows from the inequality
 $\dim \A(\overline{s})_{\overline{k}}\leq \sum \dim \phi_L(\A(\overline{s})^L_{\overline{k}})$,
 where the last sum is taken over all equivalence classes of $L\vdash \overline{k}$.
\end{proof}

\begin{lem}\label{lemma 2}
  Let $\A'(\overline{s})$ be the subalgebra of $S$ generated
 by $\{F_k^{\mu}(\overline{s})\}_{k\geq 1}^{\mu\in \CC}$. Then $\A'(\overline{s})\subset \A(\overline{s})$.
\end{lem}

\begin{proof}
  $\ $

  It suffices to show
 $F_{k_1,\ldots,k_r}^{\mu_1,\ldots,\mu_r}(\overline{s}):=F_{k_1}^{\mu_1}(\overline{s})\star \cdots \star F_{k_r}^{\mu_r}(\overline{s})\in \A(\overline{s})$
 for any $r,\ k_i\geq 1$, and $\mu_i\in \CC^*$. The case of $r=1$ has
 been already treated in Lemma~\ref{lemma 0}. The arguments for
 general $r$ are similar. Choose any $a\leq b$, such that $[a;b]\leq
 k\delta$, where $k:=k_1+\cdots+k_r$. We can further assume $a=0$.
  Let us consider any summand from the definition of $F_{k_1,\ldots,k_r}^{\mu_1,\ldots,\mu_r}(\overline{s})$
 with $\overline{l}:=[a;b]$ variables being multiplied by $\xi$.
  We will check that as $\xi$ tends to $\infty$ or $0$, both limits exist and differ by the constant $s_a\ldots s_b$.

  For a fixed summand as above, define $\{\overline{l}^t\}_{t=1}^r\in \ZZ_+^{[n]}$
 satisfying $\overline{l}=\overline{l}^1+\cdots+\overline{l}^r$
 by considering those variables $x_{i,j}$ which are multiplied by
 $\xi$ and get substituted into $F_{k_t}^{\mu_t}(\overline{s})$.
  Following the proof of Lemma~\ref{lemma 0},
 the function $F_{k_t}^{\mu_t}(\overline{s})_\xi^{\overline{l}^t}$ grows at
 the  speed $\xi^{\sum_{i\in [n]} l^t_i(l^t_{i+1}-l^t_i-1)+\sum_{i\in [n]}\mathrm{max}\{l^t_i,l^t_{i+1}\}}$ as $\xi\to \infty$
 and at the  speed $\xi^{\sum_{i\in [n]} l^t_i(-l^t_{i+1}+l^t_i-1)+\sum_{i\in [n]} \mathrm{min}\{l^t_i,l^t_{i+1}\}}$ as $\xi\to 0$.
  To estimate these powers, we note that $(a-b)(a-b-1)\geq 0$ for any $a,b\in \ZZ$, implying
   $$\mathrm{min}(a,b)+\frac{a^2+b^2-a-b}{2}\geq ab$$
 with equality holding if and only if $a-b\in \{-1,0,1\}$. Therefore,
 $$\sum_{i\in [n]} l^t_i(l^t_{i+1}-l^t_i-1)+\sum_{i\in [n]}\mathrm{max}\{l^t_i,l^t_{i+1}\}\leq 0\leq
   \sum_{i\in [n]} l^t_i(-l^t_{i+1}+l^t_i-1)+\sum_{i\in [n]} \mathrm{min}\{l^t_i,l^t_{i+1}\},$$
 with equalities holding if and only if $l^t_i-l^t_{i+1}\in \{\pm 1,0\}$ for any $i\in [n]$.
  Since the limits of
   $$\omega_{i,j}(\xi\cdot x,y),\ \omega_{i,j}(x,\xi \cdot y),\ \omega_{i,j}(\xi\cdot x,\xi \cdot y)\ \mathrm{as}\ \xi \to 0,\infty\ \mathrm{exist}\
     \forall\ i,j\in [n],$$
 the limits of the corresponding summands in the symmetrization are
 well defined as either $\xi\to 0,\infty$. Moreover, they are both
 zero if $|l^t_i-l^t_{i+1}|>1$ for some $1\leq t\leq r, i\in [n]$.

  Assuming finally that $|l^t_i-l^t_{i+1}|\leq 1$ for any $t,i$,
 the formulas from the proof of Lemma~\ref{lemma 0} imply that the
 ratio of the limits as $\xi$ goes to $\infty$ and $0$ equals to
  $$\prod_{t=1}^r\prod_{i\in [n]}\left(\frac{s_0\ldots s_i}{-\mu_t}\right)^{l^t_i-l^t_{i+1}}=\prod_{i\in [n]} (s_0\ldots s_i)^{l_i-l_{i+1}}=s_a\ldots s_b.$$
 The result follows.
\end{proof}

\begin{lem}\label{lemma 3}
  For any $k\in \NN$, we have $\dim \A'(\overline{s})_{k\delta}\geq \dim \mathcal{R}_{k}$.
\end{lem}

\begin{proof}
 $\ $

  Choose any pairwise distinct $\mu_1,\ldots,\mu_n\in \CC$
 and consider a subspace $\A''(\overline{s})$ of $\A'(\overline{s})$ spanned by
 $F_{k_1,\ldots,k_r}^{\mu_{i_1},\ldots,\mu_{i_r}}(\overline{s})$
 with $r\geq 0, k_1\geq k_2\geq \cdots \geq k_r>0$, and $1\leq i_1,\ldots,i_r\leq n$.
  It suffices to show
   $$\dim \A''(\overline{s})_{k\delta}\geq \dim \mathcal{R}_{k}.$$
 For a Young diagram $\lambda$, we introduce the specialization map
   $$\varphi_{\lambda}: S_{|\lambda|\cdot \delta}\to \CC(\{y_{i,j}\}_{i\in [n]}^{1\leq j\leq l(\lambda)})$$
 by specializing the variables $x_{i,j}$ as follows
   $$x_{i,\lambda_1+\cdots+\lambda_{t-1}+j}\mapsto q^{2j}y_{i,t}\ \ \mathrm{for\ any}\ \ 1\leq t\leq l(\lambda), 1\leq j\leq \lambda_t, i\in [n].$$
  It is clear that for any $\overline{k}=(k_1,\ldots,k_r)$ with $\sum_i k_i=k=|\lambda|$ and $\overline{k}> \lambda'$
 (here $>$ denotes the lexicographic order on Young diagrams and $\lambda'$ denotes the transposed to $\lambda$ Young diagram),
 we have $\varphi_{\lambda}(F_{\overline{k}}^{\overline{\mu}}(\overline{s}))=0$ for all $\overline{\mu}\in \CC^r$.
  Therefore, it remains to prove
   $$\sum_{\overline{k}\vdash k} \dim \varphi_{\overline{k}'}(\mathrm{span}\{F_{\overline{k}}^{\overline{\mu}}(\overline{s})|\overline{\mu}\in \CC^r\})
     \geq \dim \mathcal{R}_{k}.$$
 Let us first consider the case $k_1=\cdots=k_r\Rightarrow k=rk_1$. Then
  $$\varphi_{\overline{k}'}(F_{\overline{k}}^{\overline{\mu}}(\overline{s}))=
    Z\cdot \prod_{t=1}^r \prod_{i\in [n]} (s_0\ldots s_i \prod_{j=1}^{k_1} y_{i,j}-\mu_t \prod_{j=1}^{k_1} y_{i+1,j})$$
 for a certain nonzero common factor $Z$. Define $Y_i:=y_{i,1}\cdots y_{i,k_1}$.
 Since the polynomials
  $$f_t(Y_1,\ldots,Y_n):=\prod_{i\in [n]}(s_0\ldots s_iY_i-\mu_t Y_{i+1}),\ 1\leq t\leq n,$$
 are algebraically independent, we immediately get the required
 dimension estimate for this particular $\overline{k}$.
  The general case follows immediately.
\end{proof}

  By Lemmas~\ref{lemma 1}--\ref{lemma 3}, the subspace $\A(\overline{s})$
 is generated by $F_k^{\mu}(\overline{s})$ and has
 the prescribed dimensions of each $\ZZ_+^{[n]}$-graded component.

\begin{lem}\label{lemma 4}
  The algebra $\A(\overline{s})$ is commutative. Moreover, for any
 $\overline{\mu}=(\mu_1,\ldots,\mu_n)\in \CC^n$ with $\mu_i\ne \mu_j$ for $i\ne j$, there is an
 isomorphism $\mathcal{R}\iso \A(\overline{s})$ given by $T_{i,k}\mapsto F_k^{\mu_i}(\overline{s})$.
\end{lem}

\begin{proof}
  $\ $

  It suffices to prove $F_{m_1}^{\nu_1}(\overline{s})\star F_{m_2}^{\nu_2}(\overline{s})=F_{m_2}^{\nu_2}(\overline{s})\star F_{m_1}^{\nu_1}(\overline{s})$
 for any $m_1,m_2\in \NN$ and $\nu_1,\nu_2\in \CC$. Define $F:=F_{m_1,m_2}^{\nu_1,\nu_2}(\overline{s})-F_{m_2,m_1}^{\nu_2,\nu_1}(\overline{s})$.
  Due to previous lemmas, $F$ can be written as a certain linear
 combination of $F_{\overline{k}}^{\overline{\mu}}(\overline{s})$
 with $\overline{k}=(k_1\geq k_2\geq \cdots)$.

  We claim that $\varphi_{(2,1^{m_1+m_2-2})}(F)=0$.
  Together with the properties of $\varphi_\lambda$ discussed above,
 this equality implies $F=\sum_{r=1}^n \pi_r\cdot F_{m_1+m_2}^{\mu_r}(\overline{s})$
 for some $\pi_r\in \CC$. Let us multiply both sides of
 this equality by $\prod_{i\in [n]}\prod_{1\leq j,j'\leq m_1+m_2} (x_{i,j}-x_{i+1,j'})$
 and consider a specialization $x_{i,j}\mapsto y_i\ \forall i,j$.
  The left-hand side will clearly specialize to $0$, while the right-hand side
 will specialize to
  $$\prod_{i\in [n]} ((1-q^{-2})y_i)^{(m_1+m_2)(m_1+m_2-1)}\cdot
    \sum_{r=1}^n \left\{\pi_r\cdot \prod_{i\in [n]} (s_0\ldots s_i y_i^{m_1+m_2}-\mu_r y_{i+1}^{m_1+m_2})\right\}.$$
 This expression vanishes if and only if $\pi_1=\cdots=\pi_n=0$, and so $F=0$ as required.

  Finally, let us prove the equality $\varphi_{(2,1^{m_1+m_2-2})}(F)=0$.
 The statement is obvious when either $m_1$ or $m_2$ is zero.
 To prove for general $m_1,m_2>0$, we can assume by induction that
  $$F_{m_1'}^{\nu_1'}(\overline{s}')\star F_{m_2'}^{\nu_2'}(\overline{s}')=F_{m_2'}^{\nu_2'}(\overline{s}')\star F_{m_1'}^{\nu_1'}(\overline{s}')$$
 for any $m_1'<m_1, m_2'<m_2$, $\nu_1', \nu_2'\in \CC$, and $\prod_i s_i'=1$ (though $\{s_i'\}$ are not necessarily generic).

 By straightforward computations, we find
  $\varphi_{(2,1^{m_1+m_2-2})}(F_{m_1,m_2}^{\nu_1,\nu_2}(\overline{s}))=\mathrm{Sym}(A_1\cdot B_1)$,
 where the symmetrization is taken with respect to all permutations of $\{y_{i,j}\}_{i\in [n]}^{2\leq j\leq m_1+m_2-1}$
 preserving index $i$,  $A_1\in \CC(\{y_{i,j}\})$
 is symmetric, while $B_1$ is given by the
 following explicit formula
\begin{multline*}
  B_1=\frac{\prod_{i\in [n]}\prod_{2\leq j\ne j'\leq m_1} (y_{i,j}-q^{-2}y_{i,j'}) \cdot \prod_{i\in [n]}\prod_{m_1< j\ne j'< m_1+m_2} (y_{i,j}-q^{-2}y_{i,j'})}
           {\prod_{i\in [n]}\prod_{2\leq j\ne j'\leq m_1} (y_{i,j}-y_{i+1,j'})\cdot \prod_{i\in [n]}\prod_{m_1< j\ne j'< m_1+m_2}(y_{i,j}-y_{i+1,j'})}\times\\
      \prod_{i,i'\in [n]} \prod_{2\leq j\leq m_1}\prod_{m_1<j'<m_1+m_2} \omega_{i,i'}(y_{i,j}/y_{i',j'})\times\\
      \prod_{i\in [n]}(s_0\ldots s_iy_{i,1}\prod_{j=2}^{m_1}y_{i,j}-\nu_1y_{i+1,1}\prod_{j=2}^{m_1}y_{i+1,j})
                    (s_0\ldots s_iy_{i,1}\prod_{j=m_1+1}^{m_1+m_2-1}y_{i,j}-\nu_2y_{i+1,1}\prod_{j=m_1+1}^{m_1+m_2-1}y_{i+1,j})\\
  \Rightarrow \mathrm{Sym}(B_1)=\kappa \cdot(F_{m_1-1}^{\nu_1'}(\overline{s}')\star F_{m_2-1}^{\nu_2'}(\overline{s}'))
  (y_{0,2},\ldots, y_{0,m_1+m_2-1};\ldots; y_{n-1,2},\ldots,y_{n-1, m_1+m_2-1})
\end{multline*}
 with
 $\nu_1':=\nu_1\cdot \frac{y_{0,1}}{y_{n-1,1}},\ \nu_2':=\nu_2\cdot \frac{y_{0,1}}{y_{n-1,1}},\ s_i':=s_i\cdot \frac{y_{i,1}^2}{y_{i-1,1}y_{i+1,1}},\
  \kappa:=\prod_{i\in [n]} \frac{y_{i,1}^2y_{n-1,1}^2}{y_{0,1}^2}$.

  Permuting $m_1\leftrightarrow m_2, \nu_1 \leftrightarrow \nu_2$, we  get
$$
\varphi_{(2,1^{m_1+m_2-2})}(F_{m_2,m_1}^{\nu_2,\nu_1}(\overline{s}))=
   \kappa\cdot A_1\cdot (F_{m_2-1}^{\nu_2'}(\overline{s}')\star F_{m_1-1}^{\nu_1'}(\overline{s}'))(y_{0,2},\ldots, y_{n-1,m_1+m_2-1}).
$$
  Applying the induction assumption, we find
   $\varphi_{(2,1^{m_1+m_2-2})}(F)=\kappa A_1 [F_{m_1-1}^{\nu_1'}(\overline{s}'),F_{m_2-1}^{\nu_2'}(\overline{s}')]=0.$
 This proves the inductive step and, hence, completes the proof of the claim.
\end{proof}

  The results of Theorem~\ref{main1} follow immediately by combining the above four lemmas.

\begin{rem}
 The proof of Lemma~\ref{lemma 1} implies
$\A(\overline{s})=\CC$
 for any $s_0,\ldots,s_{n-1}\in \CC^*$ such that $\prod_i s_i^{\alpha_i}\notin
 q^\ZZ\cdot d^\ZZ$ unless $\alpha_0=\ldots=\alpha_{n-1}=0$.
\end{rem}

       %%%%%%%%%%%%%%% Relation to Negut results. Shuffle realization of the horizontal Heisenberg Uq(gl_1) %%%%%%%%%%%%%%%

\subsection{Shuffle realization of $\dot{U}^{\mathrm{h}}(\gl_n)^+$ and $\dot{U}^{\mathrm{h}}(\gl_1)^+$}\label{section_horizontal_heisenberg}
$\ $

  In~\cite{N2}, the author introduced the notion of the \emph{slope filtration} on  $S$.
 For a zero slope, the corresponding subspace $A^0\subset S$ is $\ZZ_+^{[n]}$-graded
 with the graded component $A^0_{\overline{k}}$ given by
  $$F\in A^0_{\overline{k}}\Longleftrightarrow
    F\in S_{\overline{k},0} \ \mathrm{and}\ \exists \underset{\xi\to \infty}\lim F^{\overline{l}}_\xi\ \ \forall\ 0\leq \overline{l}\leq \overline{k}.$$
 While proving Theorem~\ref{Negut theorem}, the author obtained the
 following description of $A^0$:

\begin{prop}\cite[Lemma 4.4]{N2}\label{Negut isomorphism}

\noindent
 (a) The isomorphism $\Psi: \ddot{U}^+\iso S$ identifies $\dot{U}^{\mathrm{h}}(\gl_n)^+$ with $A^0$.

\noindent
 (b) Under the isomorphism $\Psi^{\mathrm{h}}: \dot{U}^{\mathrm{h}}(\gl_n)^+\iso A^0$ from (a),
    the image $X_k:=\Psi^{\mathrm{h}}(h^{\mathrm{h}}_k)$ of the $k$th generator
    $h^{\mathrm{h}}_k\in \dot{U}^{\mathrm{h}}(\gl_1)^+\subset \dot{U}^{\mathrm{h}}(\gl_n)^+$
    is uniquely (up to a constant) characterized by
     $$X_k\in S_{k\delta,0}\ \mathrm{and}\  \underset{\xi\to \infty}\lim (X_k)^{\overline{l}}_\xi=0\ \ \forall\ 0< \overline{l}< k\delta.$$
\end{prop}

  This proposition provides a shuffle characterization of both $\dot{U}^{\mathrm{h}}(\gl_n)^+$ and $\dot{U}^{\mathrm{h}}(\gl_1)^+$.

\noindent
 In particular, we immediately obtain the following result:

\begin{thm}\label{main2}
  We have $\Psi^{-1}(\A(\overline{s}))\subset \dot{U}^{\mathrm{h}}(\gl_n)^+$ for generic $\{s_i\}$ such that $s_0\ldots s_{n-1}=1$.
\end{thm}

\begin{proof}
 $\ $

  By Theorem~\ref{main1} and Proposition~\ref{Negut isomorphism}(a), it suffices to show that $F^\mu_k(\overline{s})\in  A^0$.
 The latter is equivalent to the existence of limits
 $\underset{\xi\to \infty}\lim (F^\mu_k(\overline{s}))^{\overline{l}}_\xi$ for all $0\leq \overline{l}\leq k\delta$.
  As $\xi\to \infty$, the function
  $(F^\mu_k(\overline{s}))^{\overline{l}}_\xi$ grows at the speed $\xi^{\sum_{i\in [n]} l_i(l_{i+1}-l_i+1)-\sum_{i\in [n]} \min\{l_i,l_{i+1}\}}$
 (see the proof of Lemma~\ref{lemma 0}). Since
  $\sum_{i\in [n]} l_i(l_{i+1}-l_i+1)-\sum_{i\in [n]} \min\{l_i,l_{i+1}\}=\sum_{i\in [n]} (l_il_{i+1}-\frac{l_i^2+l_{i+1}^2-l_i-l_{i+1}}{2}-\min\{l_i,l_{i+1}\})$
 and each summand is nonpositive (see the proof of Lemma~\ref{lemma 2}), the
 aforementioned power of $\xi$ is nonpositive as well.
 Hence, the limit $\underset{\xi\to \infty}\lim (F^\mu_k(\overline{s}))^{\overline{l}}_\xi$ does exist.
 This completes the proof.
\end{proof}

  We complete this section by providing explicit formulas for the elements $X_k=\Psi^{\mathrm{h}}(h^{\mathrm{h}}_k)\in S$
 (this answers one of the questions raised in~\cite[Section 5.6]{N2}).
   Consider the elements
  $$F_0:=\textbf{1},
    F_k:=\frac{\prod_{i\in [n]} \prod_{1\leq j\ne j'\leq k} (q^{-1}x_{i,j}-qx_{i,j'})\cdot \prod_{i\in [n]} \prod_{j=1}^k x_{i,j}}
              {\prod_{i\in [n]} \prod_{1\leq j,j'\leq k} (x_{i+1,j'}-x_{i,j})}\in S_{k\delta,0}\ \mathrm{for}\ k>0.$$
  Note that $F_k=\frac{(-q^{k-1})^{nk}}{s_0^ns_1^{n-1}\ldots s_{n-1}} \cdot F_k^0(\overline{s})\in \A(\overline{s})$ for any $\{s_i\}$ such that
 $\prod_{i\in [n]}s_i=1$.
  We also define
 $$L_k\in S_{k\delta,0}\ \mathrm{via}\  \exp\left(\sum_{k=1}^\infty L_k t^k\right)=\sum_{k=0}^\infty F_k t^k.$$
 The relevant properties of these elements are formulated in our next theorem:

\begin{thm}\label{main3}
\noindent
 (a) For $\overline{l}\notin \{0, \delta, 2\delta,\ldots, k\delta\}$,
 we have $\underset{\xi\to \infty}\lim (F_k)^{\overline{l}}_{\xi}=0$.

\noindent
 (b) For any $0\leq l\leq k$, we have
     $\underset{\xi\to \infty}\lim (F_k)^{l\delta}_{\xi}=F_l\cdot F_{k-l}$.

\noindent
 (c) For any $0<\overline{l}<k\delta$, we have
     $\underset{\xi\to \infty}\lim (L_k)^{\overline{l}}_{\xi}=0$.
\end{thm}

\begin{proof}
$\ $

 (a) For any $0\leq \overline{l}\leq k\delta$, the function
 $(F_k)^{\overline{l}}_\xi$ grows at the speed $\xi^{\sum_{i\in [n]} l_i(l_{i+1}-l_i)}$ as $\xi\to \infty$.
  Note that $\sum_{i\in [n]} l_i(l_{i+1}-l_i)=-\frac{1}{2}\sum_{i\in [n]} (l_i-l_{i+1})^2\leq 0$.
 Moreover, the equality holds if and only if $l_0=\ldots=l_{n-1}\Leftrightarrow \overline{l}\in \{0,\delta, 2\delta,\ldots\}$.
  Part (a) follows.

 (b) Straightforward.

 (c) Standard (it is actually equivalent to the general \emph{exponential} relation between group-like elements and primitive elements;
     see~\cite[Section 4.3]{N2} for the related coproduct).
\end{proof}

\begin{cor}
  Combining this result with Proposition~\ref{Negut isomorphism}(b), we
 see that $L_k$ and $X_k$ coincide up to a nonzero constant, and
 the isomorphism $\Psi^{\mathrm{h}}$ identifies $\dot{U}^{\mathrm{h}}(\gl_1)^+$ with $\CC[F_1,F_2,\ldots]$.
\end{cor}

%%%%%%%%%%%%%%%%%%%%%%%%%%%%%%%%%%%%%%% SECTION 4 %%%%%%%%%%%%%%%%%%%%%%%%%%%%%%%%%%%%%%%%%%%%%%%%%%%%%%%

\section{Bethe algebra realization of $\A(\overline{s})$}

  We provide an alternative viewpoint on the subspaces $\A(\overline{s})$ for generic $\{s_i\}$ with $\prod_{i\in [n]}s_i=1$.
 Some of the results from this section (the computation of $\phi^{\bar{u},t}_{p,\bar{c}}, \Gamma^{\bar{u},t}_{p,\bar{c}},X^{\bar{u},t}_{p,\bar{c}}$)
 are not essential for the rest of this paper, but will be used in
 the forthcoming publications in order to formulate Bethe ansatz for $\ddot{U}_{q,d}(\ssl_n)$ as well as establish connections with the results of~\cite{FKSW}.

                                 %%%%%%%%%%%%%% Saito's Vertex Representations %%%%%%%%%%%%%%

\subsection{Vertex representations $\rho_{p,\bar{c}}$}
$\ $

  Recall the algebra $\ddot{U}^{'}_{q,d}(\ssl_n)$ introduced in Theorem~\ref{Drinfeld double sln}(c).
 We start by recalling the construction of vertex $\ddot{U}^{'}_{q,d}(\ssl_n)$-representations from~\cite{S},
 which generalize the classical Frenkel--Jing construction.
  Let $S_n$ be the \emph{generalized} Heisenberg algebra generated by $\{H_{i,k}|i\in [n],k\in \ZZ\backslash\{0\}\}$ and
 a central element $H_0$ with the defining relations
  $$[H_{i,k}, H_{j,l}]=d^{-km_{i,j}}\frac{[k]_q\cdot [ka_{i,j}]_q}{k}\delta_{k,-l}\cdot H_0.$$
 Let $S_{n}^+$ be the Lie subalgebra generated by $\{H_{i,k}|i\in[n],k>0\}\sqcup \{H_0\}$,
 and let $\CC v_0$ be the $S_n^+$-representation with $H_{i,k}$ acting trivially and $H_0$ acting via the identity operator.
  The induced representation $F_n:=\mathrm{Ind}_{S_{n}^+}^{S_n}  \CC v_0$ is called the \emph{Fock representation} of $S_n$.

  We denote by $\{\bar{\alpha}_i\}_{i=1}^{n-1}$ the simple roots of $\ssl_n$,
 by $\{\bar{\Lambda}_i\}_{i=1}^{n-1}$ the fundamental weights of $\ssl_n$, by $\{\bar{h}_i\}_{i=1}^{n-1}$ the simple coroots of $\ssl_n$.
 Let $\bar{Q}:=\bigoplus_{i=1}^{n-1} \ZZ\bar{\alpha}_i$ be the root lattice of $\ssl_n$,
 $\bar{P}:=\bigoplus_{i=1}^{n-1} \ZZ\bar{\Lambda}_i=\bigoplus_{i=2}^{n-1} \ZZ{\bar{\alpha}_i}\oplus \ZZ\bar{\Lambda}_{n-1}$
 be the weight lattice of $\ssl_n$. We also set
 $$\bar{\alpha}_0:=-\sum_{i=1}^{n-1}\bar{\alpha}_i\in \bar{Q},\ \bar{\Lambda}_0:=0\in \bar{P},\ \bar{h}_0:=-\sum_{i=1}^{n-1}\bar{h}_i.$$
  Let $\CC\{\bar{P}\}$ be the $\CC$-algebra generated by
 $e^{\bar{\alpha}_2},\ldots,e^{\bar{\alpha}_{n-1}},e^{\bar{\Lambda}_{n-1}}$
 with the defining relations:
  $$e^{\bar{\alpha}_i}\cdot e^{\bar{\alpha}_j}=(-1)^{\langle \bar{h}_i,\bar{\alpha}_j \rangle}e^{\bar{\alpha}_j}\cdot e^{\bar{\alpha}_i},\
    e^{\bar{\alpha}_i}\cdot e^{\bar{\Lambda}_{n-1}}=(-1)^{\delta_{i,n-1}}e^{\bar{\Lambda}_{n-1}}\cdot e^{\bar{\alpha}_i}.$$
 For $\alpha=\sum_{i=2}^{n-1} m_i\bar{\alpha}_i+m_n\bar{\Lambda}_{n-1}$, we define $e^{\bar{\alpha}}\in \CC\{\bar{P}\}$ via
   $$e^{\bar{\alpha}}:=(e^{\bar{\alpha}_2})^{m_2}\cdots (e^{\bar{\alpha}_{n-1}})^{m_{n-1}}(e^{\bar{\Lambda}_{n-1}})^{m_n}.$$
  Let $\CC\{\bar{Q}\}$ be the subalgebra of $\CC\{\bar{P}\}$ generated by $\{e^{\bar{\alpha}_i}\}_{i=1}^{n-1}$.

 For every $0\leq p\leq n-1$, define the space
  $$W(p)_n:=F_n\otimes \CC\{\bar{Q}\}e^{\bar{\Lambda}_p}.$$
 Consider the operators $H_{i,l}, e^{\bar{\alpha}}, \partial_{\bar{\alpha}_i}, z^{H_{i,0}}, \mathrm{d}$ acting on $W(p)_n$,
 which assign to every element
  $$v\otimes e^{\bar{\beta}}=(H_{i_1,-k_1}\cdots H_{i_N,-k_N}v_0)\otimes e^{\sum_{j=1}^{n-1}m_j\bar{\alpha}_j+\bar{\Lambda}_p}\in W(p)_n$$
 the following values:
\begin{equation*}
 H_{i,l}(v\otimes e^{\bar{\beta}}):=(H_{i,l}v)\otimes e^{\bar{\beta}},\
 e^{\bar{\alpha}}(v\otimes e^{\bar{\beta}}):=v\otimes e^{\bar{\alpha}}e^{\bar{\beta}},\
 \partial_{\bar{\alpha}_i}(v\otimes e^{\bar{\beta}}):=\langle \bar{h}_i,\bar{\beta} \rangle v\otimes e^{\bar{\beta}},
\end{equation*}
\begin{equation*}
   z^{H_{i,0}}(v\otimes e^{\bar{\beta}}):=
   z^{\langle \bar{h}_i,\bar{\beta} \rangle} d^{\frac{1}{2}\sum_{j=1}^{n-1}\langle \bar{h}_i,m_j\bar{\alpha}_j \rangle m_{i,j}} v\otimes e^{\bar{\beta}},
\end{equation*}
\begin{equation*}
   \mathrm{d}(v\otimes e^{\bar{\beta}}):=-(\sum k_i+((\bar{\beta},\bar{\beta})-(\bar{\Lambda}_p,\bar{\Lambda}_p))/2)v\otimes e^{\bar{\beta}}.
\end{equation*}
  The following result provides a natural structure of an $\ddot{U}^{'}_{q,d}(\ssl_n)$-module on $W(p)_n$.

\begin{prop}\cite[Proposition 3.2.2]{S}\label{Saito representation}
  For any $\bar{c}=(c_0,\ldots, c_{n-1})\in (\CC^*)^n$ and $0\leq p\leq n-1$, the following formulas define an action of $\ddot{U}^{'}_{q,d}(\ssl_n)$ on
 $W(p)_n$ (which does not depend on $\bar{c}$):
\begin{equation*}
 \rho_{p,\bar{c}}(e_i(z))=
 c_i\cdot \exp\left( \sum_{k>0} \frac{q^{-k/2}}{[k]_q}H_{i,-k}z^k\right)\cdot
 \exp\left( -\sum_{k>0}\frac{q^{-k/2}}{[k]_q}H_{i,k}z^{-k} \right)
 \cdot e^{\bar{\alpha}_i}z^{H_{i,0}+1},
\end{equation*}
\begin{equation*}
 \rho_{p,\bar{c}}(f_i(z))=
 c_i^{-1}\cdot \exp\left( -\sum_{k>0} \frac{q^{k/2}}{[k]_q}H_{i,-k}z^k\right)\cdot
 \exp\left( \sum_{k>0}\frac{q^{k/2}}{[k]_q}H_{i,k}z^{-k} \right)
 \cdot e^{-\bar{\alpha}_i}z^{-H_{i,0}+1},
\end{equation*}
\begin{equation*}
 \rho_{p,\bar{c}}(\psi_i^{\pm}(z))=\exp\left( \pm(q-q^{-1})\sum_{k>0} H_{i,\pm k}z^{\mp k} \right)\cdot q^{\pm \partial_{\bar{\alpha}_i}},
\end{equation*}
\begin{equation*}
 \rho_{p,\bar{c}}(\gamma^{\pm 1/2})=q^{\pm 1/2},\ \ \rho_{p,\bar{c}}(q^{\pm d_1})=q^{\pm \mathrm{d}}.
\end{equation*}
\end{prop}

                                 %%%%%%%%%%%%%% Functional %%%%%%%%%%%%%%

\subsection{Functionals $\phi^0_{p,\bar{c}}, \phi^{\bar{u}}_{p,\bar{c}}, \phi^{\bar{u},t}_{p,\bar{c}}$ on $\ddot{U}^{'\leq}$}
$\ $

  In this subsection, we introduce and ``explicitly compute'' three functionals on $\ddot{U}^{'\leq}$.

\noindent
 $\bullet$
 \emph{Top matrix coefficient.}

 Consider the functional
  $$\phi^0_{p,\bar{c}}:\ddot{U}^{'\leq}\longrightarrow \CC\ \ \textit{defined\ by}\ \
    \phi^0_{p,\bar{c}}(A):=\langle v_0\otimes e^{\bar{\Lambda}_p}|\rho_{p,\bar{c}}(A)|v_0\otimes e^{\bar{\Lambda}_p} \rangle.$$
 Since $h_{i,j}(v_0\otimes e^{\bar{\Lambda}_p})=0$ for $j>0$, it remains to compute the values of $\phi^0_{p,\bar{c}}$ evaluated at
  $$f_{i_1,j_1}f_{i_2,j_2}\cdots f_{i_m,j_m}\psi_{0,0}^{r_0}\cdots \psi_{n-1,0}^{r_{n-1}}\cdot(\gamma^{1/2})^a(q^{d_1})^b$$
 with $a,b\in \ZZ$, $\bar{r}:=(r_0,\ldots,r_{n-1})\in \ZZ^{[n]}$ and $\sum_{s=1}^m \bar{\alpha}_{i_s}=0\in \bar{Q}$.
  The latter condition means that the multiset $\{i_1,\ldots,i_m\}$
 contains an equal number of each of the indices $\{0,\ldots,n-1\}$.
  Due to the defining quadratic relation (T3) of $\ddot{U}^{'}_{q,d}(\ssl_n)$,
 it suffices to compute the series
$$
  \phi^0_{p,\bar{c};N,\bar{r},a,b}(z_{0,1},\ldots,z_{n-1,N}):=
  \phi^0_{p,\bar{c}}\left(\prod_{j=1}^N (f_0(z_{0,j})\cdots f_{n-1}(z_{n-1,j}))\cdot \prod_{i\in [n]}  \psi_{i,0}^{r_i}\cdot \gamma^{a/2}q^{bd_1}\right).
$$
 In this expression, we order the $z$-variables as follows:
  $$z_{0,1},\ldots,z_{n-1,1},z_{0,2},\ldots,z_{n-1,2},\ldots,z_{0,N},\ldots,z_{n-1,N}.$$
 Normally ordering the product $\prod_{j=1}^N (f_0(z_{0,j})\cdots f_{n-1}(z_{n-1,j}))$, we get the following result:
\begin{prop}\label{computation of top functional for sln}
 For $n\geq 3$, we have:
\begin{multline*}
  \phi^0_{p,\bar{c};N,\bar{r},a,b}(z_{0,1},\ldots,z_{n-1,N})=
  (c_0\ldots c_{n-1})^{-N}q^{a/2+r_p-r_0}d^{\frac{N(n-2)}{2}}\cdot
  \prod_{j=1}^N \frac{z_{0,j}}{z_{p,j}}\times \\
  \frac{\prod_{i\in [n]}\prod_{1\leq j< j'\leq N} (z_{i,j}-z_{i,j'})(z_{i,j}-q^2z_{i,j'})\cdot \prod_{i\in [n]}\prod_{j=1}^N z_{i,j}}
       {\prod_{i\in [n]}\prod_{1\leq j\leq j'\leq N} (z_{i,j}-qdz_{i+1,j'})\cdot
        \prod_{i\in [n]}\prod_{1\leq j< j'\leq N}    (z_{i,j}-qd^{-1}z_{i-1,j'})}.
\end{multline*}
\end{prop}

\noindent
 $\bullet$
 \emph{Top level graded trace.}

  Recall the operator $\mathrm{d}$ acting diagonally in the natural basis of $W(p)_n$.
 Clearly all its eigenvalues are in $-\ZZ_+$.
 Let $M(p)_n:=\mathrm{Ker}(\mathrm{d})\subset W(p)_n$ be its kernel.
 The following is obvious:

\begin{lem}\label{zero level}
 (a) The subspace $M(p)_n$ is $U_q(\ssl_n)$-invariant and is isomorphic to the irreducible highest weight $U_q(\ssl_n)$-module $L_q(\bar{\Lambda}_p)$.

\noindent
 (b) For any $\bar{\sigma}=\{1\leq \sigma_1<\sigma_2<\cdots<\sigma_p\leq n\}$,
 let $\bar{\Lambda}_p^{\bar{\sigma}}$ be the $\ssl_n$-weight having entries $1-\frac{p}{n}$ at the places $\{\sigma_i\}_{i=1}^p$ and $-\frac{p}{n}$ elsewhere.
 Then $\{v_0\otimes e^{\bar{\Lambda}_p^{\bar{\sigma}}}\}_{\bar{\sigma}}$ form a basis of $M(p)_n$.
\end{lem}

  Define the degree operators $\mathrm{d}_1,\ldots,\mathrm{d}_{n-1}$ acting on $W(p)_n$ by
   $$\mathrm{d}_r(v\otimes e^{\sum_{j=1}^{n-1} m_j\bar{\alpha}_j + \bar{\Lambda}_p})=
     -m_r\cdot v\otimes e^{\sum_{j=1}^{n-1} m_j\bar{\alpha}_j + \bar{\Lambda}_p}\ \ \forall\ v\in F_n.$$
  For any $\bar{u}=(u_1,\ldots,u_{n-1})\in (\CC^*)^{n-1}$, consider the functional
   $$\phi^{\bar{u}}_{p,\bar{c}}:\ddot{U}^{'\leq}\longrightarrow \CC \ \ \textit{defined\ by}\ \ \phi^{\bar{u}}_{p,\bar{c}}(A):=
     \sum_{\bar{\sigma}}
     \langle v_0\otimes e^{\bar{\Lambda}_p^{\bar{\sigma}}}| \rho_{p,\bar{c}}(A)u_1^{\mathrm{d}_1}\cdots u_{n-1}^{\mathrm{d}_{n-1}}|
     v_0\otimes e^{\bar{\Lambda}_p^{\bar{\sigma}}}\rangle,$$
  computing the \emph{$\bar{Q}$-graded} trace of the $A$-action on the subspace $M(p)_n$
  (here $u_i^{\mathrm{d}_i}$ makes sense as $\mathrm{d}_i$ acts with integer  eigenvalues).
  Since $h_{i,j}(v_0\otimes e^{\bar{\Lambda}_p^{\bar{\sigma}}})=0$ for $j>0$, it suffices to compute the generating series
$$
  \phi^{\bar{u}}_{p,\bar{c};N,\bar{r},a,b}(z_{0,1},\ldots,z_{n-1,N}):=
  \phi^{\bar{u}}_{p,\bar{c}}\left(\prod_{j=1}^N (f_0(z_{0,j})\cdots f_{n-1}(z_{n-1,j}))\cdot \prod_{i\in [n]} \psi_{i,0}^{r_i}\cdot \gamma^{a/2}q^{bd_1}\right).
$$

 Normally ordering the product $\prod_{j=1}^N (f_0(z_{0,j})\cdots f_{n-1}(z_{n-1,j}))$, we get the following result:

\begin{prop}\label{computation of partial trace functional for sln}
  For $n\geq 3$, we have:
\begin{multline*}
 \phi^{\bar{u}}_{p,\bar{c};N,\bar{r},a,b}(z_{0,1},\ldots,z_{n-1,N})=
  (c_0\ldots  c_{n-1})^{-N}q^{a/2}d^{\frac{N(n-2)}{2}}\times \\
  \frac{\prod_{i\in [n]} \prod_{1\leq j< j'\leq N} (z_{i,j}-z_{i,j'})(z_{i,j}-q^2z_{i,j'})}
       {\prod_{i\in [n]} \prod_{1\leq j\leq j'\leq N} (z_{i,j}-qdz_{i+1,j'}) \cdot
        \prod_{i\in [n]} \prod_{1\leq j< j'\leq N} (z_{i,j}-qd^{-1}z_{i-1,j'})}\times\\
   (-1)^p\prod_{j=1}^p\frac{1}{u_1\ldots u_{j-1}}\cdot [\mu^p]
   \left\{\prod_{i\in [n]} \left(\prod_{j=1}^N z_{i+1,j}-\mu u_1\ldots u_i  q^{r_{i+1}-r_i}\prod_{j=1}^N z_{i,j}\right)\right\},
\end{multline*}
 where $[\mu^p]\{\cdots\}$ denotes the coefficient of $\mu^p$ in $\{\cdots\}$.
\end{prop}

\noindent
 $\bullet$
 \emph{Full graded trace.}

 Finally, we introduce the most general functional
   $$\phi^{\bar{u},t}_{p,\bar{c}}:\ddot{U}^{'\leq}\longrightarrow \CC[[t]] \ \ \textit{defined\ by}\ \ \phi^{\bar{u},t}_{p,\bar{c}}(A):=
     \tr_{W(p)_n} (\rho_{p,\bar{c}}(A)u_1^{\mathrm{d}_1}\cdots u_{n-1}^{\mathrm{d}_{n-1}}t^{-\mathrm{d}}),$$
 computing the \emph{$\bar{Q}\times \ZZ_+$-graded trace} of the $A$-action on the representation $W(p)_n$.
 Due to the quadratic relations and the $\bar{Q}$-grading, it suffices to compute the following generating series:
\begin{multline*}
 \phi^{\bar{u},t}_{p,\bar{c};N,\bar{k},\bar{r},a,b}(z_{0,1},\ldots,z_{n-1,N};w_{0,1},\ldots,w_{0,k_0},\ldots,w_{n-1,1},\ldots,w_{n-1,k_{n-1}}):=\\
 \phi^{\bar{u},t}_{p,\bar{c}}\left( \prod_{j=1}^N (f_0(z_{0,j})\cdots f_{n-1}(z_{n-1,j}))\cdot \prod_{i\in [n]} \prod_{j=1}^{k_i}  \bar{\psi}_i^+(w_{i,j})
 \cdot \prod_{i\in [n]} \psi_{i,0}^{r_i}\cdot \gamma^{a/2}q^{bd_1}\right).
\end{multline*}
 In what follows, $(z;t)_\infty$ is defined by $(z;t)_\infty:=\prod_{a=0}^\infty (1-t^az)$.

\begin{thm}\label{main4}
  For $n\geq 3$, we have:
\begin{multline}\tag{$\diamond$}
  \phi^{\bar{u},t}_{p,\bar{c};N,\bar{k},\bar{r},a,b}(z_{0,1},\ldots,z_{n-1,N};w_{0,1},\ldots,w_{n-1,k_{n-1}})=
  (c_0\ldots  c_{n-1})^{-N}q^{a/2}d^{\frac{N(n-2)}{2}}\times \\
  \frac{\prod_{i\in [n]} \prod_{1\leq j< j'\leq N} (z_{i,j}-z_{i,j'})(z_{i,j}-q^2z_{i,j'})\cdot \prod_{i\in [n]}\prod_{j=1}^N z_{i,j}}
       {\prod_{i\in [n]} \prod_{1\leq j\leq j'\leq N} (z_{i,j}-qdz_{i+1,j'})\cdot
        \prod_{i\in [n]} \prod_{1\leq j< j'\leq N} (z_{i,j}-qd^{-1}z_{i-1,j'})}\times\\
  q^{r_p-r_0}\cdot \prod_{j=1}^N \frac{z_{0,j}}{z_{p,j}}\cdot \theta(\vec{y}; \bar{\Omega})\times\\
  \frac{1}{(T;T)_\infty^n}\cdot \prod_{i\in [n]}\prod_{a,b=1}^N
  \frac{(T\cdot\frac{z_{i,a}}{z_{i,b}};T)_\infty \cdot (Tq^2\frac{z_{i,a}}{z_{i,b}};T)_\infty}
       {(Tqd\frac{z_{i+1,a}}{z_{i,b}};T)_\infty \cdot (Tqd^{-1}\frac{z_{i-1,a}}{z_{i,b}};T)_\infty}  \times\\
  \prod_{i\in [n]}\prod_{a=1}^N\prod_{b=1}^{k_i}
  \frac{(Tq^2\frac{q^{1/2}z_{i,a}}{w_{i,b}};T)_\infty \cdot (Tq^{-1}d\frac{q^{1/2}z_{i+1,a}}{w_{i,b}};T)_\infty
         \cdot (Tq^{-1}d^{-1}\frac{q^{1/2}z_{i-1,a}}{w_{i,b}};T)_\infty}
       {(Tq^{-2}\frac{q^{1/2}z_{i,a}}{w_{i,b}};T)_\infty \cdot (Tqd\frac{q^{1/2}z_{i+1,a}}{w_{i,b}};T)_\infty
         \cdot (Tqd^{-1}\frac{q^{1/2}z_{i-1,a}}{w_{i,b}};T)_\infty},
\end{multline}
 where $T:=\frac{t}{q^b}$ and
  $\theta(\vec{y},\bar{\Omega}):=\sum_{\vec{n}\in \ZZ^{n-1}}\exp (2\pi \sqrt{-1}(\frac{1}{2}\vec{n}\bar{\Omega}\vec{n}'+\vec{n}\vec{y}'))$
 is the classical Riemann theta function with
  $\bar{\Omega}=\frac{1}{2\pi \sqrt{-1}}\cdot(a_{i,j}\ln(T))_{i,j=1}^{n-1}$ and
  $$\vec{y}=(y_1,\ldots,y_{n-1})\ \mathrm{with}\ y_i=\frac{1}{2\pi \sqrt{-1}}\ln\left(u_i^{-1}T^{\delta_{p,i}}q^{2r_i-r_{i-1}-r_{i+1}}
    \prod_{j=1}^N \frac{z_{i-1,j}z_{i+1,j}}{z_{i,j}^2}\right).$$
\end{thm}

 We start with the following two auxiliary results:

\begin{lem}\label{change of basis}
  The matrix $\left(\frac{d^{-km_{i,j}}[k]_q[ka_{i,j}]_q}{k}\right)_{i\in [n]}^{j\in [n]}$ is nondegenerate if and only if $q^{2k},q^kd^{\pm k}\ne 1$.
\end{lem}

  Therefore if $q^2, dq,d^{-1}q$ are not roots of unity,
 we can choose a new basis  $\{\wt{H}_{i,-k}\}_{i\in [n]}$
 of the space $\mathrm{span}_\CC\{H_{0,-k},\ldots,H_{n-1,-k}\}$,
 such that $[H_{i,k},\wt{H}_{j,-l}]=\delta_{i,j}\delta_{k,l}H_0$ for any $i,j\in [n],\ k,l\in \NN$.
  In particular, the elements $\{H_{i,k},\wt{H}_{i,-k},H_0\}_{k>0}$
 form a Heisenberg Lie algebra $\h_i$ for any $i\in [n]$, and $\h_i$ commutes with $\h_j$ for any $i\ne j\in [n]$.

\begin{lem}\label{useful lemma}
  Let $\aaa$ be a Heisenberg Lie algebra with the basis $\{a_k\}_{k\in \ZZ}$ and the commutator relation $[a_k,a_l]=\delta_{k,-l}\lambda_ka_0$.
 Consider the Fock $\aaa$-representation $F:=\mathrm{Ind}_{\aaa_+}^{\aaa} \CC v_0$ with the central charge $a_0=1$
 and the degree operator $d\in \mathrm{End}(F)$ satisfying $[d,a_{k}]=ka_k$ and $d(v_0)=0$.
 Then:
  $$\tr_F \left\{\exp\left(\sum_{j=1}^\infty x_ja_{-j}\right)\cdot\exp\left(\sum_{j=1}^\infty y_ja_j\right)\cdot t^{-d}\right\}=
    \frac{1}{(t;t)_\infty} \cdot \exp\left(\sum_{j=1}^\infty\frac{x_jy_j\lambda_jt^j}{1-t^j}\right)\  \forall\ x_j,y_j\in \CC.$$
\end{lem}

\begin{proof}
$\ $

  Applying the formula $\langle a_{-j}^lv_0|a_{-j}^k a_j^k| a_{-j}^lv_0 \rangle=l(l-1)\cdots (l-k+1)\lambda_j^k$, we get
\begin{multline*}
 \tr_F \left\{\exp\left(\sum_{j=1}^\infty x_ja_{-j}\right)\exp\left(\sum_{j=1}^\infty y_ja_j\right)\cdot t^{-d}\right\}=
 \sum_{k_1,k_2,\ldots\geq 0} \tr_F \left(\prod_{j=1}^\infty\frac{(x_jy_j)^{k_j}}{(k_j!)^2} a_{-j}^{k_j}a_j^{k_j}  \cdot t^{-d}\right)=\\
 \prod_{j=1}^\infty\left\{ \sum_{k_j=0}^\infty\sum_{l_j=k_j}^\infty\frac{(x_jy_j)^{k_j}}{(k_j!)^2}\cdot
                            \frac{l_j!\cdot \lambda_j^{k_j}}{(l_j-k_j)!}\cdot t^{jl_j}\right\}=
 \prod_{j=1}^\infty \left\{\sum_{k_j=0}^\infty \frac{(x_jy_j\lambda_jt^j)^{k_j}}{k_j!}\cdot \frac{1}{(1-t^j)^{k_j+1}}\right\}.
\end{multline*}
 The result follows.
\end{proof}

\begin{proof}[Proof of Theorem~\ref{main4}]
$\ $

 Reordering the factors of
  $\prod_{j=1}^N (f_0(z_{0,j})\cdots f_{n-1}(z_{n-1,j}))\cdot \prod_{i\in [n]} \prod_{j=1}^{k_i} \bar{\psi}_i^+(w_{i,j})
   \cdot \prod_{i\in [n]} \psi_{i,0}^{r_i}$
 in the normal order, we gain the product of factors from the
 first two lines of ($\diamond$).
  The $\bar{Q}\times \ZZ_+$-graded trace of the normally ordered
 product splits as $\tr_1\cdot \tr_2$, where
  $$\tr_1=\tr_{\CC\{\bar{Q}\}e^{\bar{\Lambda}_p}}
     \left( q^{\sum_{i\in [n]} r_i\partial_{\bar{\alpha}_i}}\cdot \prod_{i\in [n]}
     \prod_{j=1}^N z_{i,j}^{-H_{i,0}}\cdot \prod_{i=1}^{n-1} u_i^{d_i}\cdot (t/q^b)^{\mathrm{d}^{(2)}}\right),$$
  $$\tr_2= \tr_{F_n}
    \left( \exp\left(\sum_{i\in [n]}\sum_{k>0}u_{i,k}H_{i,-k}\right)\cdot
    \exp\left(\sum_{i\in [n]}\sum_{k>0}(v^{(1)}_{i,k}+v^{(2)}_{i,k})H_{i,k}\right)\cdot (t/q^b)^{\mathrm{d}^{(1)}} \right)$$
 with
  $$u_{i,k}:=\frac{-q^{k/2}}{[k]_q}\sum_{j=1}^N z_{i,j}^k,\ v^{(1)}_{i,k}:=\frac{q^{k/2}}{[k]_q}\sum_{j=1}^N z_{i,j}^{-k},\
    v^{(2)}_{i,k}:=(q-q^{-1})\sum_{j=1}^{k_i} w_{i,j}^{-k}$$
 and the operators $\mathrm{d}^{(1)}\in \mathrm{End}(F_n),\mathrm{d}^{(2)}\in \mathrm{End}(\CC\{\bar{Q}\}e^{\bar{\Lambda}_p})$ defined by
  $$\mathrm{d}^{(1)}(H_{i_1,-k_1}\cdots H_{i_l,-k_l}v_0)=\sum_{i=1}^l k_i\cdot H_{i_1,-k_1}\cdots H_{i_l,-k_l}v_0,\
    \mathrm{d}^{(2)}(e^{\bar{\beta}})=\frac{(\bar{\beta},\bar{\beta})-(\bar{\Lambda}_p,\bar{\Lambda}_p)}{2}\cdot e^{\bar{\beta}}.$$

  The computation of $\tr_1$ is straightforward, and we get exactly the expression from the third line of ($\diamond$).
 To evaluate $\tr_2$, we rewrite
   $\sum_{i\in [n]}\sum_{k>0}u_{i,k}H_{i,-k}=\sum_{i\in [n]}\sum_{k>0}\wt{u}_{i,k}\wt{H}_{i,-k}$
 with $\wt{H}_{i,-k}$ defined right after Lemma~\ref{change of basis} and
  $\wt{u}_{i,k}=\sum_{i'\in [n]} d^{-km_{i,i'}}\frac{[k]_q[ka_{i,i'}]_q}{k}u_{i',k}.$
 The commutativity of $\h_i$ and $\h_j$ for $i\ne j$ allows us to
 rewrite $\tr_2$ as a product of the corresponding traces over the
 $\h_i$-Fock modules.
 Applying Lemma~\ref{useful lemma}, we see (after routine computations)
 that $\tr_2$ is equal to the product of the factors from the last two lines in ($\diamond$).
\end{proof}

%%%%%%%%%%%%%%%%%%%%%%%%%%%%%%%%%%%%%%% From functionals to transfer matrices %%%%%%%%%%%%%%%%%%%%%%%%%%%%%%%%%%%%%%%%%%%%%%%%%%%%%%%

\subsection{Functionals via pairing}
$\ $

  Recall the Hopf algebra pairing $\varphi': \ddot{U}^{'\geq}\times \ddot{U}^{'\leq}\to \CC$ from Theorem~\ref{Drinfeld double sln}.
 As $\varphi'$ is nondegenerate, there exist unique elements $X^0_{p,\bar{c}}, X^{\bar{u}}_{p,\bar{c}}\in \ddot{U}^{'\geq,\wedge}$ and
 $X^{\bar{u},t}_{p,\bar{c}}\in \ddot{U}^{'\geq,\wedge}[[t]]$
 such that
   $$\phi^0_{p,\bar{c}}(X)=\varphi'(X^0_{p,\bar{c}}, X),\
     \phi^{\bar{u}}_{p,\bar{c}}(X)=\varphi'(X^{\bar{u}}_{p,\bar{c}},X),\
     \phi^{\bar{u},t}_{p,\bar{c}}(X)=\varphi'(X^{\bar{u},t}_{p,\bar{c}},X)\ \ \ \forall\ X\in \ddot{U}^{'\leq}.$$
  The goal of this section is to find these elements explicitly.

  We will actually compute these elements in the shuffle presentation.
  In order to do this, we first extend the isomorphism $\Psi$ from Theorem~\ref{Negut theorem}
  to the isomorphism
    $$\Psi^\geq:\ddot{U}^{'\geq} \iso S^{\geq}.$$
  Here $S^{\geq}$ is generated by $S$ and the formal generators
  $\psi_{i,k} (k<0),\psi_{i,0}^{\pm 1},\gamma^{\pm 1/2}, q^{\pm d_1}$
  with the defining relations compatible with those for $\ddot{U}^{'\geq}$.
   In particular, for $F\in S_{\overline{k},d}$ we have
   $$q^{d_1}Fq^{-d_1}=q^{-d}\cdot F.$$
  We define $\Gamma^0_{p,\bar{c}}, \Gamma^{\bar{u}}_{p,\bar{c}}, \Gamma^{\bar{u},t}_{p,\bar{c}}$ as
  the images of $X^0_{p,\bar{c}}, X^{\bar{u}}_{p,\bar{c}},X^{\bar{u},t}_{p,\bar{c}}$ under the isomorphism $\Psi^\geq$, respectively.
  Now we are ready to state the main result of this section:

\begin{thm}\label{main5}

 We have the following formulas:

\medskip
\noindent
 (a)  $\Gamma^0_{p,\bar{c}}=\sum_{N=0}^\infty (c_0\ldots c_{n-1})^{-N}\cdot \Gamma^0_{p;N}\cdot q^{\bar{\Lambda}_p}q^{-d_1}$
  with $\Gamma^0_{p;N}\in  S_{N\delta}$ given by
$$
 \Gamma^0_{p;N}=
 (1-q^{-2})^{nN}(-q^nd^{-n/2})^{N^2}\cdot
 \prod_{j=1}^N\frac{x_{0,j}}{x_{p,j}}\cdot
 \frac{\prod_{i\in [n]}\prod_{j\ne j'} (x_{i,j}-q^{-2}x_{i,j'})\cdot \prod_{i\in [n]}\prod_{j=1}^N x_{i,j}}
      {\prod_{i\in [n]}\prod_{j,j'} (x_{i,j}-x_{i+1,j'})}.
$$

\medskip
\noindent
 (b) $\Gamma^{\bar{u}}_{p,\bar{c}}=\sum_{N\geq 0} (c_0\ldots c_{n-1})^{-N}\cdot \Gamma^{\bar{u}}_{p;N}\cdot q^{-d_1}$
 with $\Gamma^{\bar{u}}_{p;N}\in  S_{N\delta}^\geq$ given by
\begin{multline*}
   \Gamma^{\bar{u}}_{p;N}=
   (1-q^{-2})^{nN}(-q^nd^{-n/2})^{N^2}\prod_{j=1}^p\frac{1}{u_1\ldots u_{j-1}}\times \\
   \frac{\prod_{i\in [n]}\prod_{j\ne j'} (x_{i,j}-q^{-2}x_{i,j'})}
        {\prod_{i\in [n]}\prod_{j,j'} (x_{i,j}-x_{i+1,j'})}\cdot
   (-1)^p [\mu^p]\left\{\prod_{i\in [n]} \left(\prod_{j=1}^N x_{i+1,j}-\mu u_1\ldots u_i\prod_{j=1}^N x_{i,j}\cdot
                                               q^{\bar{\Lambda}_{i+1}-\bar{\Lambda}_i}\right)\right\},
\end{multline*}
 where in the last product we take all $x_{i,j}$ to the left and all $q^{\bar{\Lambda}_{i+1}-\bar{\Lambda}_i}$ to the right.

\medskip
\noindent
 (c) $\Gamma^{\bar{u},t}_{p,\bar{c}}=\sum_{N\geq 0} (c_0\ldots c_{n-1})^{-N}\cdot \Gamma^{\bar{u},t}_{p;N}\cdot q^{\bar{\Lambda}_p}q^{-d_1}$ with
 $\Gamma^{\bar{u},t}_{p;N}\in  S_{N\delta}^\geq$ given by
\begin{multline*}
 \Gamma^{\bar{u},t}_{p;N}=
 (1-q^{-2})^{nN}(-q^nd^{-n/2})^{N^2}\times\\
 \frac{\prod_{i\in [n]}\prod_{j\ne j'} (x_{i,j}-q^{-2}x_{i,j'})\cdot \prod_{i\in [n]}\prod_{j=1}^N x_{i,j}}
      {\prod_{i\in [n]}\prod_{j,j'} (x_{i,j}-x_{i+1,j'})}\cdot
       \prod_{j=1}^N \frac{x_{0,j}}{x_{p,j}}\cdot  \theta(\vec{x},\wt{\Omega})\times   \\
 \frac{1}{(\bar{t};\bar{t})_\infty^n}\cdot \prod_{i\in [n]}\prod_{a,b=1}^N
   \frac{(\bar{t}\frac{x_{i,a}}{x_{i,b}};\bar{t})_\infty\cdot (\bar{t}q^2\frac{x_{i,a}}{x_{i,b}};\bar{t})_\infty}
        {(\bar{t}qd\frac{x_{i+1,a}}{x_{i,b}};\bar{t})_\infty\cdot (\bar{t}qd^{-1}\frac{x_{i-1,a}}{x_{i,b}};\bar{t})_\infty}
   \cdot \prod_{k>0}\prod_{i\in [n]}\prod_{a=1}^N   \bar{\psi}_i^-(\bar{t}^kq^{1/2}x_{i,a}),
\end{multline*}
 where $\bar{t}=t\gamma,\ \wt{\Omega}=\frac{1}{2\pi\sqrt{-1}}\cdot (a_{i,j}\ln(\bar{t}))_{i,j=1}^{n-1}$, and
  $$\vec{x}=(x_1,\ldots,x_{n-1})\ \mathrm{with}\
    x_i=\frac{1}{2\pi \sqrt{-1}}\ln\left(u_i^{-1}\bar{t}^{\delta_{p,i}}\psi_{i,0}\prod_{j=1}^N \frac{x_{i-1,j}x_{i+1,j}}{x_{i,j}^2}\right).$$
 In the above products, we take all $x_{i,j}$ to the left and all $\psi_{i,j}$ to the right.
\end{thm}

  The proof of this theorem follows by combining
 Proposition~\ref{computation of top functional for sln}, Proposition~\ref{computation of partial trace functional for sln}
 and Theorem~\ref{main4} with the following technical lemma:

\begin{lem}\label{pairing sln properties}
 (a) For any elements $a\in \ddot{U}^+, a'\in \ddot{U}^\geq \cap \ddot{U}^0, b\in \ddot{U}^-, b'\in \ddot{U}^\leq\cap \ddot{U}^0$, we have
       $$\varphi(aa', bb')=\varphi(a,b)\cdot \varphi(a',b').$$

\medskip
\noindent
 (b) For any $k_i,k'_i\in \ZZ_+$ and $A,B,C,A',B',C',a_i,b_i\in \ZZ$, we have
\begin{multline*}
 \varphi\left(\prod_{i\in [n]}\prod_{a=1}^{k_i} \bar{\psi}_i^-(z_{i,a})\prod_{i\in [n]}\psi_{i,0}^{a_i} \gamma^{A/2}q^{Bd_1}q^{Cd_2},
              \prod_{j\in [n]} \prod_{b=1}^{k'_j} \bar{\psi}_j^+(w_{j,b})\prod_{j\in [n]}\psi_{j,0}^{a'_j}\gamma^{A'/2}q^{B'd_1}q^{C'd_2} \right)=\\
  q^{-\frac{1}{2}A'B-\frac{1}{2}AB'+C'\sum a_i+C\sum a'_i+\sum_{i,j} a_ia'_ja_{i,j}}\cdot
  \prod_{i\in [n]}^{j\in [n]} \prod_{a=1}^{k_i} \prod_{b=1}^{k'_j} \frac{w_{j,b}-q^{a_{i,j}}d^{m_{i,j}}z_{i,a}}{w_{j,b}-q^{-a_{i,j}}d^{m_{i,j}}z_{i,a}}.
\end{multline*}

\medskip
\noindent
 (c) For $\bar{r}=(r_0,\ldots,r_{n-1}), \bar{s}=(s_0,\ldots,s_{n-1})\in \ZZ_+^{[n]}$ and
 elements $X\in \ddot{U}^+, Y\in \ddot{U}^-$ of the form
 $$X=e_{0,a^0_1}\cdots e_{0,a^0_{r_0}}\cdots e_{n-1,a^{n-1}_1}\cdots e_{n-1,a^{n-1}_{r_{n-1}}},\
   Y=f_{0,b^0_1}\cdots f_{0,b^0_{s_0}}\cdots f_{n-1,b^{n-1}_1}\cdots f_{n-1,b^{n-1}_{s_{n-1}}},$$
  the pairing $\varphi(X,Y)$ is expressed by an integral formula similar to~\cite[Proposition 3.10]{N2}:
 $$\varphi(X,Y)=\delta_{\bar{r},\bar{s}}
   \int \frac{(q-q^{-1})^{-\sum r_i} u_{0,1}^{b^{0}_1}\ldots u_{n-1,s_{n-1}}^{b^{n-1}_{s_{n-1}}} \Psi(X)(u_{0,1},\ldots,u_{n-1,r_{n-1}})}
   {\prod_i \prod_{j<j'}  \omega_{i,i}(u_{i,j}/u_{i,j'}) \cdot \prod_{i<i'}\prod_{j,j'} \omega_{i,i'}(u_{i,j}/u_{i',j'})}
   \prod_{i\in [n]}\prod_{j=1}^{s_i} \frac{du_{i,j}}{2\pi \sqrt{-1} u_{i,j}}.$$

                    %%%%%%%%%%%%%%%%%%%%%%%%%%%%%%%%%%%%%%%%%%%%%
                    %%%%%%%%%%%%%%%%%% REMARK #4 %%%%%%%%%%%%%%%%
                    %%%%%%%%%%%%%%%%%%%%%%%%%%%%%%%%%%%%%%%%%%%%%
     %  Actually, in the proof of Theorem 3.8, we need the following consequence of Lemma 3.9(c):
     % For any element $X\in S_{\bar{r}}$ as in (c), we have
     %\begin{multline*}
     %  \varphi(X,F_0(z_{0,1})\cdots F_0(z_{0,r_0})\cdots F_{n-1}(z_{n-1,1})\cdots  F_{n-1}(z_{n-1,r_{n-1}}))=\\
     %    \frac{\Psi(X)(z_{0,1},\ldots,z_{0,r_0},\ldots,z_{n-1,1},\ldots,z_{n-1,r_{n-1}})}{(q-q^{-1})^{\sum r_i}}
     %    \cdot \prod_{i=0}^{n-1}\prod_{j<j'}\frac{1}{\omega_{i,i}(z_{i,j}/z_{i,j'})}\cdot\prod_{i<i'}\prod_{j,j'}\frac{1}{\omega_{i,i'}(z_{i,j}/z_{i',j'})}.
     %\end{multline*}
                     %%%%%%%%%%%%%%%%%%%%%%%%%%%%%%%%%%%%%%%%%%%%%
                     %%%%%%%%%%%%%%%%%%% END %%%%%%%%%%%%%%%%%%%%%
                     %%%%%%%%%%%%%%%%%%%%%%%%%%%%%%%%%%%%%%%%%%%%%

\end{lem}

%%%%%%%%%%%%%%%%%%%%%%%%%%%%%%%%%%%%%%% Bethe subalgebras %%%%%%%%%%%%%%%%%%%%%%%%%%%%%%%%%%%%%%%%%%%%%%%%%%%%%%%

\subsection{Bethe incarnation of $\A(\overline{s})$}
$\ $

   Recalling the notion of a transfer matrix from Section~\ref{section_bethe}, it is easy to see that
  $$X^{\bar{u},t}_{p,\bar{c}}=T_{\rho_{p,\bar{c}}}(u_1^{-\bar{\Lambda}_1}\ldots u_{n-1}^{-\bar{\Lambda}_{n-1}} t^{-d_1})
    \cdot \prod_{j=1}^{n-1} u_j^{\langle\bar{\Lambda}_j,\bar{\Lambda}_p\rangle},$$
  which provides a more elegant definition of $X^{\bar{u},t}_{p,\bar{c}}$.
  Moreover, the elements $X^{\bar{u}}_{p,\bar{c}}$ can be thought of as
  certain \emph{truncations} of $X^{\bar{u},t}_{p,\bar{c}}$ obtained
  by setting $t\to 0$, while $X^{0}_{p,\bar{c}}$ are obtained by
  setting further $u_1,\ldots,u_{n-1}\to 0$.

   The commutativity of the Bethe subalgebras
  implies the commutativity of $\{\Gamma^{\bar{u},t}_{p,\bar{c}}|p,\bar{c}\}$ and
  hence of $\{\Gamma^{\bar{u},t}_{p;N}\}_{0\leq p\leq n-1}^{N\geq 1}$.
  As a result, we get the commutativity of the families
  $\{\Gamma^0_{p;N}\}_{0\leq p\leq n-1}^{N\geq 1}$ and $\{\Gamma^{\bar{u}}_{p;N}\}_{0\leq p\leq n-1}^{N\geq 1}$.
   Due to Theorem~\ref{main5}(b), the
  elements $\Gamma^{\bar{u}}_{p;N}$ have the same form as the generators of the subalgebra $\A(s_0,\ldots,s_{n-1})$
  from Section 2 with $s_i\in \CC^*\cdot e^{\bar{P}}$ given by
     $$s_i:=u_i\cdot q^{\bar{\Lambda}_{i+1}-2\bar{\Lambda}_i+\bar{\Lambda}_{i-1}}\ \mathrm{for\ all}\ i\in [n],\ \mathrm{where}\ u_0:=1/(u_1\ldots u_{n-1}).$$
   Since $e^h\ (h\in \bar{P})$ commute with $\oplus_k S_{k\delta}$, we see that those $\{s_i\}$ can be treated as
  formal parameters with $s_0\ldots s_{n-1}=1$ and $\{s_i\}$
  being generic for any choice of $\{u_i\}$.

   Finally, let us notice that while $\ddot{U}_{q,d}(\ssl_n)$
  contained the horizontal copy of $U_q(\widehat{\gl}_n)$, the
  algebra $\ddot{U}^{'}_{q,d}(\ssl_n)$ contains a horizontal copy of
  $U_q(L\gl_n)$ (that is no $q^{\pm d_2}$ and with trivial central charge $c'=0$).
   The subspace $M(p)_n$ is $U_q(L\gl_n)$-invariant and is just the $p$th fundamental representation.
  By standard results, $U_q(L\gl_n)$ admits a double construction similar to the one for $\ddot{U}^{'}_{q,d}(\ssl_n)$.
   Combining all the previous discussions with the construction
  of the universal $R$-matrices for $\ddot{U}^{'}_{q,d}(\ssl_n)$ and $U_q(L\gl_n)$, we get the following result:

\begin{thm}\label{main6}
  The Bethe subalgebra of $U_q(L\gl_n)$, corresponding to the group-like element
 $x=u_1^{-\bar{\Lambda}_1}\ldots u_{n-1}^{-\bar{\Lambda}_{n-1}}$ and
 the category of finite-dimensional $U_q(L\gl_n)$-representations,
 can be identified with
  $\A(\{u_i\cdot q^{\bar{\Lambda}_{i+1}-2\bar{\Lambda}_i+\bar{\Lambda}_{i-1}}\}_{i\in [n]})$, where $u_0:=1/(u_1\ldots u_{n-1})$.
\end{thm}

\begin{rem}
 (a) The commutativity of $\{\Gamma^0_{p;N}\}_{0\leq p\leq n-1}^{N\geq 1}$ implies that the family
      $$\left\{\prod_{j=1}^N\frac{x_{0,j}}{x_{p,j}}\cdot \frac{\prod_{i\in [n]}\prod_{j\ne j'} (x_{i,j}-q^{-2}x_{i,j'})\cdot \prod_{i\in [n]}\prod_{j=1}^N x_{i,j}}
        {\prod_{i\in [n]}\prod_{j,j'} (x_{i,j}-x_{i+1,j'})}\right\}_{0\leq p\leq n-1}^{N\geq 1}$$
     of elements from $S$ is commutative. It is easy to see that the
     subalgebra they generate is the limit algebra of
     $\A(s_0,s_1,\ldots,s_{n-1})$ as $s_1,\ldots,s_{n-1}\to 0,\ s_0=1/(s_1\ldots s_{n-1}),\ \mathrm{and}\ \{s_i\}\ \mathrm{stay\ generic}$.

\noindent
 (b)  The commutative algebras generated by $\{\Gamma^{\bar{u},t}_{p;N}\}_{0\leq p\leq n-1}^{N\geq 1}$
     can be viewed as one-parameter deformations of the algebras $\A(\bar{s})$.
     They play a crucial role in the Bethe ansatz for $\ddot{U}_{q,d}(\ssl_n)$.
\end{rem}

%%%%%%%%%%%%%%%%%%%%%%%%%%%%%%%%%%%%%%%%%%%%%%%% n=1,2 case %%%%%%%%%%%%%%%%%%%%%%%%%%%%%%%%%%%%%%%%%%%%%%%%%%%%%%%%%%%

\section{Generalizations to $n=1$ and $n=2$}

  It turns out that all the previous results of this paper can be
 actually generalized to the $n=1,2$ cases.
  The goal of this last section is to explain the required slight
 modifications.

\subsection{$n=1$ case}
$\ $

  The quantum toroidal algebra $\ddot{U}_{q,d}(\gl_1)$ has been extensively studied in the last few years.
 Roughly speaking, one just needs to modify the quadratic relations
 from Section 1.1 by replacing
  $$g_{a_{i,j}}(t)\rightsquigarrow \frac{(q_1t-1)(q_2t-1)(q_3t-1)}{(t-q_1)(t-q_2)(t-q_3)},\ \mathrm{where}\ q_1:=q^2,\ q_2:=q^{-1}d,\ q_3:=q^{-1}d^{-1},$$
 and by replacing the Serre relations (T7.1, T7.2) by
  $$\underset{z_1,z_2,z_3}\Sym\ \frac{z_2}{z_3}\cdot[e(z_1),[e(z_2),e(z_3)]]=0,$$
  $$\underset{z_1,z_2,z_3}\Sym\ \frac{z_2}{z_3}\cdot[f(z_1),[f(z_2),f(z_3)]]=0.$$
 Analogously to the $n\geq 3$ case, the map $e_i\mapsto x^i$ extends to the isomorphism $\ddot{U}_{q,d}(\gl_1)^+\iso S^{\mathrm{sm}}$.

  The results of Section 2 recover the same commutative algebra $\A^{\mathrm{sm}}$ we started from.
 On the other hand, we can apply the constructions of Section 3 to the
 \emph{Fock} $\ddot{U}^{'}_{q,d}(\gl_1)$-representations $\{F_c\}_{c\in \CC^*}$ (introduced in~\cite[Proposition A.6]{FS}).
 As a result, we will get:

\noindent
 $\circ$ The elements $\Gamma^0_c$ (corresponding to the top matrix coefficient functional $\phi^0_c$)
 are given by

  $\Gamma^0_c=\sum_{N=0}^{\infty} c^{-N}q^{-N(N-1)}\cdot K_N(x_1,\ldots,x_N)\cdot  q^{-d_1}.$

\noindent
 $\circ$ The elements $\Gamma^t_c$ (corresponding to the full graded trace functional $\phi^t_c$)
 are given by
  $$\Gamma^t_c=\sum_{N=0}^\infty \frac{c^{-N}q^{-N(N-1)}}{(\bar{t};\bar{t})_\infty}\cdot
    K_N\cdot \prod_{a,b=1}^N
    \frac{(\bar{t}\frac{x_a}{x_b};\bar{t})_\infty \cdot (\bar{t}q^2\frac{x_a}{x_b};\bar{t})_\infty}
         {(\bar{t}qd^{-1}\frac{x_a}{x_b};\bar{t})_\infty \cdot (\bar{t}qd\frac{x_a}{x_b};\bar{t})_\infty}
    \cdot \prod_{k>0}\prod_{a=1}^N \bar{\psi}^-(\bar{t}^kq^{1/2}x_a)\cdot  q^{-d_1}.$$

\subsection{$n=2$ case}
$\ $

  For $n=2$, we need first to redefine both the quantum toroidal and the shuffle algebras.

\noindent
 $\circ$ \emph{Quantum toroidal algebra of $\ssl_2$}.

  One needs to slightly modify the defining relations (T0.1--T7.2) of $\ddot{U}_{q,d}(\ssl_2)$ (see~\cite{FJMM1}).
 The function $g_{a_{i,j}}(z)$ from the relations (T1, T2, T3, T5, T6) should be changed as follows:
  $$g_{a_{i,i}}(z)\rightsquigarrow \frac{q^2z-1}{z-q^2},\ g_{a_{i,i+1}}(z) \rightsquigarrow \frac{(dz-q)(d^{-1}z-q)}{(qz-d)(qz-d^{-1})},$$
 while the cubic Serre relations (T7.1, T7.2) should be replaced with quartic Serre relations
  $$\underset{z_1,z_2,z_3}\Sym\  [e_i(z_1),[e_i(z_2),[e_i(z_3),e_{i+1}(w)]_{q^2}]]_{q^{-2}}=0,$$
  $$\underset{z_1,z_2,z_3}\Sym\ [f_i(z_1),[f_i(z_2),[f_i(z_3),f_{i+1}(w)]_{q^2}]]_{q^{-2}}=0.$$

\noindent
 $\circ$ \emph{Big shuffle algebra of type $A_1^{(1)}$}.

 One needs to modify the matrix $\Omega$ used to define the $\star$ product as follows:
  $$\omega_{i,i}(z)=\frac{z-q^{-2}}{z-1},\ \omega_{i,i+1}(z)=\frac{(z-qd)(z-qd^{-1})}{(z-1)^2}.$$

\noindent
 $\circ$ \emph{Vertex representations $\rho_{p,\bar{c}}$}.

  Finally, we need to slightly modify the formulas of $\rho_{p,\bar{c}}$ from Proposition~\ref{Saito representation}:

 (i) We redefine the commutator relations of the Heisenberg algebra $S_n$ as follows:
     $$[H_{i,k}, H_{i,l}]=\frac{[k]_q\cdot [2k]_q}{k}\delta_{k,-l}\cdot H_0,\
       [H_{i,k}, H_{i+1,l}]=-(d^k+d^{-k})\frac{[k]_q\cdot [k]_q}{k}\delta_{k,-l}\cdot H_0.$$

 (ii) We also redefine the operator $z^{H_{i,0}}$ via
     $$z^{H_{i,0}}(v\otimes e^{\bar{\beta}}):=
       z^{\langle \bar{h}_i,\bar{\beta} \rangle} v\otimes e^{\bar{\beta}}.$$

\noindent
 Once the above modifications are made, all the results from Sections 2 and 3 still hold.

%%%%%%%%%%%%%%%%%%%%%%%%%%%%%%%%%%%%%%%%%%%%%%%%%% BIBLIOGRAPHY %%%%%%%%%%%%%%%%%%%%%%%%%%%%%%%%%%%%%%%%%%%%%%%%%%%%%%%%%%%%%%%

\end{document}